\documentclass[11pt,pdftex]{amsart}

\usepackage[text={6in,8.5in}]{geometry}

\usepackage{amsmath}
\usepackage{amsfonts}    
\usepackage{amssymb}
\usepackage{texdraw}
\usepackage{epsfig}
\usepackage{epstopdf}
\usepackage{graphicx,color, url}
\usepackage{bbm}

\newtheorem{theorem}{Theorem}
  \newtheorem{lemma}[theorem]{Lemma}
  
  \newtheorem{problem}[theorem]{Problem}
  
   \newtheorem{corollary}[theorem]{Corollary}
  \newtheorem{proposition}[theorem]{Proposition}

  \theoremstyle{definition}
  
  \newtheorem{example}[theorem]{Example}
   \newtheorem{remark}[theorem]{Remark}

\newcommand{\CC}{\mathbb C} 
\newcommand{\RR}{\mathbb R}
\newcommand{\QQ}{\mathbb Q}
\newcommand{\ZZ}{\mathbb Z}

\newcommand{\TP}{\mathbb{TP}}

\newcommand{\cS}{\mathcal S}
\newcommand{\cV}{\mathcal V}
\newcommand{\cP}{\mathcal P}
\newcommand{\cF}{\mathcal F}
\newcommand{\ord}{\text{ord}}
\newcommand{\val}{\text{val}}

\newcommand{\F}{\mathcal{F}}

\newcommand{\eins}{\mathbbm{1}}

\newcommand{\suchthat}{\,:\,}

\DeclareMathOperator{\Div}{Div} 
\DeclareMathOperator{\Prin}{Prin} 
\DeclareMathOperator{\Pic}{Pic}
\DeclareMathOperator{\Sym}{Sym}

\DeclareMathOperator{\supp}{supp}
\DeclareMathOperator{\dist}{dist}
\DeclareMathOperator{\tconv}{tconv}
\DeclareMathOperator{\link}{link}
\DeclareMathOperator{\image}{image}
\DeclareMathOperator{\lcm}{lcm}
\DeclareMathOperator{\type}{type}

\date{\today}

 \title[Linear Systems on Tropical Curves]{Linear Systems on Tropical Curves}
 \author{Christian Haase}
 \address{Math. Inst. \\ FU Berlin}
 \email{haase@math.fu-berlin.de}
 
 \author{Gregg Musiker}
\address{Massachusetts Institute of Technology,
    Department of Mathematics,
    Cambridge, MA 02139}
\email {musiker@math.mit.edu}

\author{Josephine Yu}
\address{Mathematical Sciences Research Institute,
    17 Gauss Way,
    Berkeley, CA 94720}
\email{jyu@msri.org}

 

\begin{document}
 \maketitle
 
 \begin{abstract}
 
   A tropical curve $\Gamma$ is a metric graph with possibly unbounded
   edges, and tropical rational functions are continuous piecewise
   linear functions with integer slopes. We define the complete linear
   system $|D|$ of a divisor $D$ on a tropical curve $\Gamma$
   analogously to the classical counterpart. We investigate the
   structure of $|D|$ as a cell complex and show that linear systems
   are quotients of tropical modules, finitely generated by vertices
   of the cell complex.  
   Using a finite set of generators, $|D|$ defines a map from $\Gamma$
   to a  tropical projective space, and the image can be extended to a
   tropical curve of degree equal to $\deg(D)$. The tropical
   convex  hull of the image realizes the linear system $|D|$ as a
   polyhedral complex.  We show that curves for which the canonical
   divisor is not very ample are hyperelliptic.  We also show that the
   Picard group of a $\QQ$-tropical curve is a direct limit of
   critical groups of finite graphs converging to the curve. 
 \end{abstract}
 
 \section{Introduction}

An abstract tropical curve $\Gamma$ is a connected metric graph with
possibly unbounded edges.  A {\em divisor} $D$ on $\Gamma$ is a formal
(finite) $\ZZ$-linear combination $D = \sum_{x \in \Gamma} D(x) \cdot
x$ of points of $\Gamma$.  The {\em degree} of a divisor is the sum of the coefficients,
$\sum_{x} D(x)$.  The divisor is {\em effective} if
$D(x) \ge 0$ for all $x \in \Gamma$; in this case we write $D \ge 0$. 
We call $\supp(D) = \{x \in \Gamma \suchthat D(x) \neq 0\}$ the
support of the divisor $D$.

A {\em (tropical) rational function} $f$ on $\Gamma$ is a continuous
function $f : \Gamma \rightarrow \RR$ that is piecewise-linear on each
edge with finitely many pieces and integral slopes.  The {\em order}
$\ord_x(f)$ of $f$ at a point $x \in \Gamma$ is the sum of outgoing
slopes at $x$.  The {\em principal divisor} associated to $f$ is
$$
(f) := \sum_{x\in\Gamma} \ord_x(f) \cdot x.
$$
A point $x \in \Gamma$ is called a {\em zero} of $f$ if $\ord_x(f) >
0$ and a {\em pole} of $f$ if $\ord_x(f) < 0$.  We call two divisors
$D$ and $D'$ linearly equivalent and write $D \sim D'$ if $D-D'=(f)$
for some $f$.  For any divisor $D$ on $\Gamma$, let  $R(D)$ be the set
of all rational functions $f$ on $\Gamma$ such that the divisor $D +
(f)$ is effective, and $ |D| = \{ D + (f) :  f \in R(D)\},$ the {\em
  linear system} of $D$.  Let $\eins$ denote the set of constant functions on $\Gamma$.
 
The set $R(D)$ is naturally embedded in the set $\RR^\Gamma$ of all
real-valued functions on $\Gamma$, and $|D|$ is a subset of the
$d^{th}$ symmetric product of $\Gamma$ where $d = \deg(D)$.  The map
$R(D)/\eins \rightarrow |D|$ given by $f \mapsto D + (f)$ is a homeomorphism
from $R(D)/\eins$ to $|D|$.  It was shown in
\cite{GathmannKerber,MikhalkinZharkov} that $|D|$ is a cell complex, so is $R(D) / \eins$.  Our aim is to
study the combinatorial and algebraic structure of this object
$R(D)$.

In Section \ref{sec:ChipFiring} we give definitions and state linear equivalence in terms of weighted chip firing moves, which are continuous analogues 
of the chip firing games on finite graphs.  In Section \ref{sec:Gens} we show that $R(D)$ is a finitely generated tropical semi-module and describe a generating set.  In Section \ref{sec:Cell}, we study the cell complex structure of $|D|$.  We show that the vertex set of $|D|$ coincides with the generating set of $R(D)$ described in Section \ref{sec:Gens}.  We give a triangulation of the link of each cell as the order complex of a poset of possible weighted chip firing moves.  

Any set $\cF$ of linearly equivalent divisors induces a map $\phi_\cF$
from the abstract curve to a tropical projective space.  This map is
described in Section \ref{sec:Map} and we show that the tropical
convex hull of the image of this map is piecewise linear isomorphic to
$|D|$.  The image of this map $\phi_\cF$ can be naturally extended to
an embedded tropical curve, and we show in Section \ref{sec:Deg} that
the embedded curve has the same degree as the divisor that we start
with.  In Section \ref{sec:Canonical}, we show that every divisor of
positive degree is ample.  We also show that if the canonical divisor
is not very ample, then the tropical curve is hyperelliptic (but not
vice versa).  Finally in Section \ref{sec:Picard}, we show that the
Picard group of a $\QQ$-tropical curve is the direct limit of Picard
groups for finite graphs obtained by subdividing the edges.

\medskip

{\bf Acknowledgments.}  The work was inspired by discussions at the
Oberwolfach Workshop on Tropical Geometry that took place in December
2007.  Gregg Musiker and Josephine Yu were partially supported by the
National Science Foundation postdoctoral research fellowship.
Christian Haase was supported by Emmy Noether grant HA 4383/1 of the
German Research Foundation (DFG).
We thank  Matt Baker, Michael Kerber, David Speyer, and Bernd
Sturmfels for helpful discussions. 

 \begin{figure}
 \includegraphics[scale = 0.35]{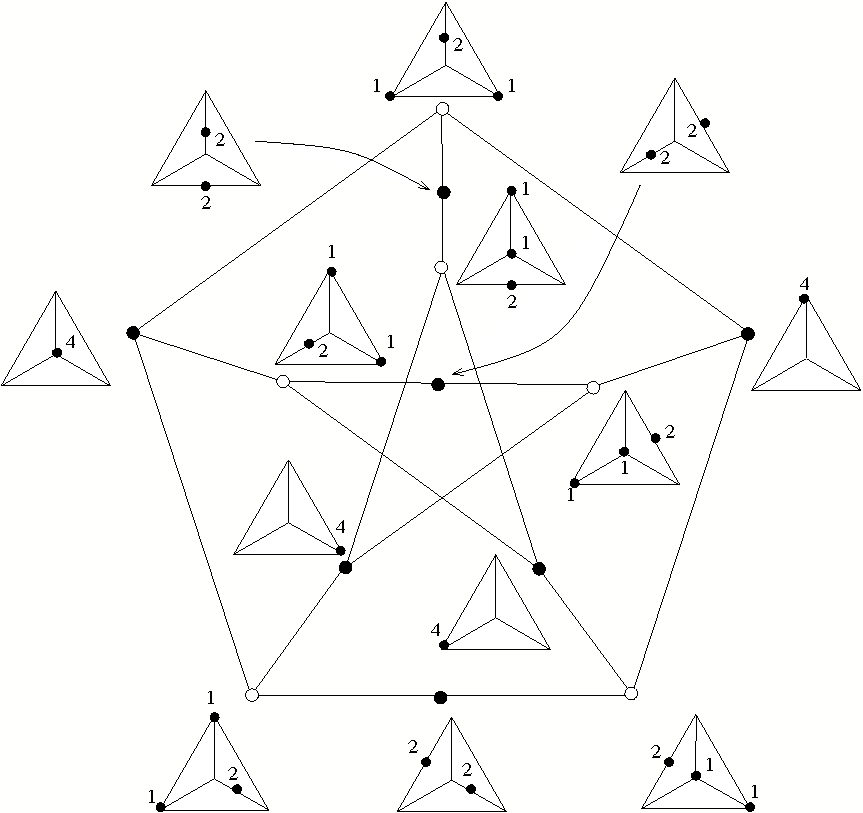}
 \caption{Let tropical curve $\Gamma$ be the complete graph on four vertices with equal length edges.  Let $K$ be the canonical divisor.  The 13 divisors shown here, together with $K$, correspond to the elements of $\cS$ that generate $R(K)$, from Theorem \ref{thm:fingen}.  The seven black dots in the Petersen graph correspond to extremals.  }
 \label{fig:gensK4}
 \end{figure}

\section{Metric graphs, rational functions, and chip-firing}
\label{sec:ChipFiring}

We begin with some notation from \cite{BF}.  We define a {\em weighted
  graph} $G$ to be a finite connected graph with vertex set $V(G)=
\{v_1,v_2,\dots, v_n\}$, edge set $E(G) = \{e_1,e_2,\dots, e_m\}$, and
a collection of positive weights $\{w_{e_1},w_{e_2},\dots, w_{e_m}\}$
associated to the edges of $G$. 
We define the {\em length} of an edge $e$ to be $L_e = \frac{1}{w_e}$.

A {\em metric graph} $\Gamma$ is a compact connected metric space such
that each point $x \in \Gamma$ has a neighborhood $U_x$ isometric to a
star-shaped set of valence $n_x \geq 1$ endowed with the path metric.
To be precise, a star-shaped set of valence $n_x$ is a set of the form
$$S(n_x,r_x) = \{z \in \CC : z = te^{2\pi /n_x} \mathrm{~for~some~} 0
\leq t < r_x \mathrm{~and~}k\in\ZZ\}.$$ 
The points $x \in \Gamma$ with valence different from $2$ are
precisely those where $\Gamma$ fails to look locally like an open
interval.  Accordingly, we refer to a point of valence $2$ as a {\em
  smooth} point. 

Let $V(\Gamma)$ be any finite nonempty subset of $\Gamma$ such that
$V(\Gamma)$ contains all of the points with $n_x \not = 2$. Then
$\Gamma \setminus V(\Gamma)$ is a finite disjoint union of open
intervals. 

For a metric graph $\Gamma$, we say that a choice of such $V(\Gamma)$
gives rise to a {\em model} $G(\Gamma)$ for $\Gamma$.  In particular,
we can define a weighted graph $G= G(\Gamma)$ from this data, letting
$V(G) = V(\Gamma)$ and $E(G)$ be the connected components of $\Gamma
\backslash V(\Gamma)$.  Each edge has a nonzero length inherited from
the metric space $\Gamma$, and we define $w_e = \frac{1}{L_e}$.

Let $V_0(\Gamma) = \{x \in \Gamma : \val(x) \neq 2\}$, where $\val$ denotes the valence of a vertex of $V(\Gamma)$.   
Unless $\Gamma$ is a circle, $V_0(\Gamma)$ gives a model.
For some of our applications, we will need to choose a model whose
vertex set is strictly bigger than $V_0(\Gamma)$.
However unless otherwise specified, the reader may assume that
$G(\Gamma)$ denotes the coarsest model and that a {\em vertex} is an
element of $V_0(\Gamma)$.

A {\em tropical curve} is a metric graph in which the leaf edges may
have length $\infty$.  A leaf edge is an edge adjacent to a one-valent
vertex.  Note that we add a ``point at infinity'' for each unbounded
edge.    A tropical rational function on a tropical curve may attain
values $\pm \infty$ at points at infinity.

We will sometimes refer to an effective divisor $D$ as a {\em chip
  configuration}.  For example, for $D = c_1 \cdot x_1+ \cdots + c_n
\cdot x_n$, we say that there are $c_i$ chips at point $x_i \in
\Gamma$.  The total number of chips is the degree of the divisor. 

We will use the term \emph{subgraph} in a topological sense, that is,
as a compact subset of tropical curve $\Gamma$ with a finite number of connected components.
 For a subgraph $\Gamma' \subset \Gamma$ and a positive real number $l$, the {\em chip firing move} $CF(\Gamma', l)$ by a (not necessarily connected) subgraph is the tropical rational function
 $CF(\Gamma', l)(x)=-\min(l,\dist(x,\Gamma'))$.
 It is constant $0$ on $\Gamma'$, has slope $-1$ in the
 $l$-neighborhood of $\Gamma'$ directed away from $\Gamma'$, and it is constant $-l$ on the rest of the graph.  The chip configuration $D +
 (CF(\Gamma', l))$ is obtained by moving one chip each of $D$ on the
 boundary of $\Gamma'$ along each edge out of the graph by distance
 $l$.  Here we assume that $l$ was chosen to be small enough so that
 the chips do not pass through each other or pass through a non-smooth
 point .  We say that a subgraph $\Gamma' \subset \Gamma$ can {\em
   fire} if for each boundary point of 
   $\Gamma' \cap \overline{\Gamma
   \backslash \Gamma'}$ there are at least as many chips as the number
 of edges pointing out of $\Gamma'$. In other words, the divisor $D +
 (CF(\Gamma', l))$ is effective for some positive real number $l$.

A tropical rational function $f$ is called a {\em weighted chip firing
  move} if there are two disjoint (not necessarily connected) proper
closed subgraphs $\Gamma_1$ and $\Gamma_2$ such that $f$ is constant
on each of them and linear (smooth) with integer slopes on the
complement.  This necessarily means that the complement of $\Gamma_1
\cup \Gamma_2$ consists only of open line segments.  In other words, a collection of $m$ chips can be fired along an edge by distance $l/m$.  A (simple) chip firing move is a special case of a weighted chip
firing move when all the slopes are $0$ or $\pm 1$.

\begin{lemma}
A
 weighted chip firing move 
 is a sum of chip firing moves.
\end{lemma}

\begin{proof} 

Let $f$ be a weighted chip firing move and $\Gamma_1$ and $\Gamma_2$
be as above.  Then  $$\Gamma \setminus (\Gamma_1 \sqcup \Gamma_2) =
L_1 \sqcup L_2 \sqcup \dots \sqcup L_k$$ where the $L_i$'s are open
line segments between $\Gamma_1$ and $\Gamma_2$. 

Suppose $f(\Gamma_2) - f(\Gamma_1) = \ell > 0$.  Consequently, the 
slope of $f$ on segment $L_i$ (viewed in the direction from 
$\Gamma_1$ to $\Gamma_2$) must be $\mathfrak{s}_i =  \ell/|L_i| \in 
\ZZ$  where $|L_i|$ is the length of edge $L_i$.
Let $\mathfrak{s}$ be the least common multiple of all the
$\mathfrak{s}_i$, and we let the $k_i$'s be the integers such that
$\mathfrak{s} =  k_i \mathfrak{s_i}$.  Then each $L_i$ has length $k_i
u_\Gamma$ where $u_\Gamma = \frac{\ell}{\mathfrak{s}}$.  
 (In other words, by setting $u_\Gamma$ to be a unit length, we obtain
 a rescaled version of $\Gamma$ such that each rescaled $L_i$ is of
 integral length.)

Along line segment $L_i$, we need to move $\mathfrak{s}_i$ chips by a
distance of $k_i u_\Gamma$. A single (unweighted) chip firing of a specific
subgraph containing $\Gamma_1$ will move a single chip a distance of 
$u_\Gamma$ along each line segment of $L_i$ towards $\Gamma_2$. 
In particular, we fire the subgraph containing $\Gamma_1$ as well as
the portion of $L_i$ traversed by chips during the previous firings
until the distance of $k_i u_{\Gamma}$ is achieved, and then repeat
$\mathfrak{s}_i$ times.
It takes $\mathfrak{s}_i k_i = \mathfrak{s}$ such moves on each
segment $L_i$.  Hence we can achieve the weighted chip firing $f$ by
performing $\mathfrak{s}$ elementary chip firings, and $f$ is the sum
of the corresponding rational functions.
\end{proof}

The following lemma makes the connection between $R(D)$ and chip
firing games. 

\begin{lemma}
Every tropical rational function is a usual sum of chip firing moves.  
\end{lemma}

\begin{proof}
Let $f$ be a tropical rational function on $\Gamma$.
Let $S$ be the finite subset of $\Gamma$ consisting of vertices of
$\Gamma$ and the corner locus (zeroes and poles) of $f$.  Let
$f(S) \subset \RR$ be the set of values of $f$ at points in $S$.
We will proceed by induction on the size of $f(S)$.  If $f(S)$ contains only one or two values, then $f$ is either a constant function or already a weighted chip firing move, so the claim is trivial.  Suppose $f(S)$ contains at least three values.  Let $c \in f(S)$ be a value that is neither the maximum or the minimum in $f(S)$.  Let $f_1, f_2$ be new tropical rational functions defined as $f_1(x) = \min(c, f(x))$ and $f_2(x) = \max(c,f(x))$ for all $x \in \Gamma$.   Then $f = f_1 + f_2 - c \eins$ where $c \eins$ denotes 
the constant function that takes value $c$ everywhere.  Let $S_1, S_2$ be the sets consisting the corner locus of $f_1, f_2$ respectively, together with the vertices of $\Gamma$.  Both $f_1(S_1)$ and $f_2(S_2)$ have strictly fewer number of values, and the assertion follows by induction.
\end{proof}

Note that even if we start with a tropical rational function $f \in R(D)$, the sequence of weighted chip firing moves $f_1, \dots, f_{n}$ for which $f = f_1 + \cdots + f_n$  may not be in $R(D)$, i.e.\ the divisors $D + (f_i)$ may not be effective although $D+(f)$ is.

The following proposition follows easily from the previous lemma.
\begin{proposition}
\label{prop:CFM}
Two divisors are linearly equivalent if and only if one can be attained from the other using 
chip firing moves.
\end{proposition}
 
 \section{Extremals and Generators of $R(D)$}
 \label{sec:Gens}
 
 The {\em tropical semiring} $(\RR, \oplus, \odot)$ is the set of real numbers $\RR$ with two tropical operations:
 $$
 a \oplus b = \max(a,b), \text{ and } a \odot b = a + b.
 $$ 
 The space $R(D)$ is naturally a subset of the space $\RR^\Gamma$ of real-valued functions on $\Gamma$.  For $f,g \in \RR^\Gamma$, and $a \in \RR$ the tropical sum $f \oplus g$ and the tropical scalar multiplication $a \odot f$ are defined by taking tropical sums and products pointwise.
  
 \begin{lemma}
 The space $R(D)$ is a tropical semi-module, i.e.\ it is closed under tropical addition and tropical scalar multiplication.  
 \end{lemma}
 
  \begin{proof}
It is clear that $(f\odot g) = (f+g) = (f) + (g)$ for tropical rational functions, thus $R(D)$ is closed under tropical scalar multiplication.  
  Let $f, g \in R(D)$ and $x \in \Gamma$.   If $f(x) > g(x)$, then $\ord_x (f\oplus g) = \ord_x(f)$.  If $f(x) < g(x)$, then $\ord_x (f\oplus g) = \ord_x(g)$.  If $f(x) = g(x)$, then 
for each direction $\overrightarrow{v}$,
the outgoing slope of $f\oplus g$ in the neighborhood of $x$ 
in the direction $\overrightarrow{v}$ is the maximum of 
each outgoing slope 
in the direction $\overrightarrow{v}$ of $f$ and $g$, so $\ord_x(f\oplus g) \geq \ord_x(f) $. 
Hence $\ord_x(f \oplus g) + D(x) > 0$ for all $x \in \Gamma$, so $f \oplus g \in R(D)$.
 \end{proof}
 
Tropical semi-modules are also called {\em tropically convex sets} \cite{DevelinSturmfels}.  Since $R(D + (f)) = R(D) + f$, the tropical algebraic structure of $R(D)$ does not depend on the choice of the representative $D$.  An element $f \in R(D)$ is called {\em extremal} if for any $g_1, g_2 \in R(D)$, $f = g_1 \oplus g_2 \implies f = g_1 \text{ or } f = g_2$.   An element $f$ is an extremal if and only if all its tropical scalar multiples $c \odot f$ also are extremals.  Any generating set of $R(D)$ must contain all extremals up to tropical scalar multiplication.

 \begin{lemma}\label{lem:ext}
A tropical rational function $f$ is an extremal of $R(D)$ if and only if there are not two proper 
subgraphs $\Gamma_1$ and $\Gamma_2$ covering $\Gamma$ ($\Gamma_1 \cup \Gamma_2 = \Gamma$) such that each can fire on $D + (f)$.
 \end{lemma}

 \begin{proof}
 Suppose there are two such graphs that can fire.  The corresponding rational functions $g_1, g_2$ can be chosen so that $g_i$ is zero on $\Gamma_i$, and they are non-positive.  Since $\Gamma_1 \cup \Gamma_2 = \Gamma$, $g_1 \oplus g_2 = 0$, so $(f + g_1) \oplus (f + g_2) = f$ and $f$ is not an extremal.
 
 Now suppose $f = g_1 \oplus g_2$ for some $g_1, g_2 \neq f$ in $R(D)$.  Let $\Gamma_i$ be the loci where $f = g_i$.  Then $\Gamma_1 \cup \Gamma_2 = \Gamma$.  Let $\varepsilon_i > 0 $ be such that $g_i$ is smooth in the $\varepsilon_i$-neighborhood of $\Gamma_i$, outside of $\Gamma_i$.  Then each $\Gamma_i$ can fire $\varepsilon_i$ distance.
 \end{proof}
  
 A {\em cut set} of a graph $\Gamma$ is a set of points $A \subset \Gamma$ such that $\Gamma \backslash A$ is not connected.  A {\em smooth cut set} is a cut set consisting of smooth points (valent 2 points).  Note that being a smooth cut set depends only on the topology of $\Gamma$ and is not affected by the choice of model $G(\Gamma)$.
   
 \begin{theorem}
 \label{thm:fingen}
Let $\cS$ be the set of rational functions $f \in R(D)$ such that the support of $D + (f)$ does not contain a smooth cut set.  Then
\begin{enumerate}
\item[(a)] $\cS$ contain all the extremals of $R(D)$,
\item[(b)] $\cS$ is finite modulo tropical scaling, and
\item[(c)] $\cS$ generates $R(D)$ as a tropical semi-module.
\end{enumerate}
 \end{theorem}

 \begin{proof}
\begin{enumerate}
\item[(a)] Suppose $f \notin \cS$, then $D + (f)$ splits $\Gamma$ into two subgraphs $\Gamma_1$ and $\Gamma_2$.   Both of these graphs can fire, and the union of their closures is the entire $\Gamma$, so by Lemma \ref{lem:ext}, $f$ is not an extremal.

 \item[(b)]Let $f \in \cS$.  Then removing the set of edges containing the support of $D + (f)$ does not disconnect $\Gamma$, the remaining edges contain a spanning tree of $\Gamma$.  There are finitely many spanning trees in a graph, and there are finitely many possible slopes for each edge in this spanning tree, because by \cite[Lemma 1.8]{GathmannKerber}, the absolute value of the slopes of $f\in R(D)$ is bounded by a constant depending only on $\Gamma$ and $D$.  Therefore, the possible values of $f$ on vertices of $\Gamma$ is finite modulo tropical scaling.
 
Furthermore, $D + (f)$ cannot have more than one zero on each edge, 
because two zeroes on the same edge form a smooth cut set.  
On each edge, knowing the values and the slopes of $f$ at the two end points uniquely determines $f$, given the fact that all the chips of $D+(f)$ must fall on the same point of a given edge. 
Hence, the values of $f$ and the outgoing slopes at all vertices uniquely determine $f$.  Since $\{f|_V: f \in R(D)\}$ is finite modulo tropical scaling and there are finitely many possible slopes at each vertex, we conclude that $\cS$ is finite modulo tropical scaling.

\item[(c)]Let $f$ be an arbitrary function in $R(D)$.  We need to show that $f$ can be written as a finite tropical sum of elements of $\cS$.  Let $N(f)$ be the number of smooth points in $\supp(D + (f))$.
If $f$ is not already in $\cS$, then there is a smooth cut set $A$ and two components $\Gamma_1$ and $\Gamma_2$.  Let $g_1$ and $g_2$ be weighted chip firing moves that that fire all chips on their boundaries for as far as possible.  Then $f = (f+g_1) \oplus (f+g_2)$.  Repeating this decomposition terminates after finite steps because $0 \leq N(f+g_i) < N(f)$ for each $i = 1,2$.
\end{enumerate}
\end{proof} 

 \begin{proposition}\label{prop:ExtGen}
 Any finitely generated tropical sub-semimodule $M$ of $\RR^\Gamma$ is generated by the extremals.
 \end{proposition}
 
 \begin{proof}
 Let $f_1, f_2, \dots, f_n$ be a generating set of $M \subset \RR^{\Gamma}$.  Suppose $f_n$ is not an extremal. Then $f_n = g \oplus h$ for some $g, h \in M$ such that $f_n \neq g$ and $f_n \neq h$.  Since $f_1, \dots, f_n$ generate $M$, we have 
 $$g = (a_1 \odot f_1) \oplus \cdots \oplus (a_{n-1} \odot f_{n-1}) \oplus (a_n \odot f_n) \text{, and }
 h = (b_1 \odot f_1) \oplus \cdots \oplus (b_{n-1} \odot f_{n-1}) \oplus (b_n \odot f_n)$$
 for some $a_1, \dots, a_n, b_1, \dots, b_n \in \RR$.   
 Since $g \leq f_n,~h \leq f_n$ pointwise, and $g \neq f_n,~h \neq f_n$, we must have $a_n < 0$ and $b_n < 0$. 
Then
 $$f_n = g \oplus f = (a_1 \odot f_1) \oplus \cdots \oplus (a_{n-1} \odot f_{n-1}) \oplus (b_1 \odot f_1) \oplus \cdots \oplus (b_{n-1} \odot f_{n-1}),$$
so $f_n$ is in the tropical semi-module generated by $f_1, \dots, f_{n-1}$.  We can remove non-extremals from any finite generating set this way, so $M$ is generated by the extremals.
 \end{proof}
 
 \begin{corollary}
 The tropical semimodule $R(D)$ is generated by extremals.  This generating set is minimal and unique up to tropical scalar multiplication.
 \end{corollary}
  
 The set of extremals can be obtained from $\cS$ by removing the elements that do not satisfy the condition in Lemma \ref{lem:ext}.
 
 \begin{example}
 \label{ex:gensK4}
 Let $\Gamma$ be a tropical curve with the complete graph on 4 vertices with equal edge lengths as a model.  Consider the canonical divisor $K$, that is the divisor with value $1$ on the four vertices and zero elsewhere.  Then the set $\cS$ from Theorem \ref{thm:fingen} consists of 14 elements, 7 of which are extremals. See Figure \ref{fig:gensK4}.   

If the edge lengths of the complete graph are not all equal, then the set $\cS$ may be different from this.  
We will describe the cell complex structure of $R(K)$ in the next section, in Example \ref{ex:cellsK4}.

 \end{example}

\section{Cell complex structure of $|D|$}
\label{sec:Cell}

As seen in the previous section, $R(D) \subset \RR^\Gamma$ is finitely generated as a tropical semi-module or a tropical polytope.  However, it is not a polyhedral complex in the ordinary sense.  For example, let $\Gamma$ be the line segment $[0,1]$, and $D$ be the point $1$.  Then $R(D)$ is the tropical convex hull of $f,g \in \RR^{\Gamma}$ where $f(x) = x$ and $g(x)=0$.  Although $R(D)$ is one-dimensional, it does not contain the usual line segment between any two points in it.  
  Letting $\eins$ denote the constant function taking the value $1$ at all points, we consider functions in $R(D)$ modulo addition of $\eins$, i.e. translation.

\begin{lemma}
\label{lem:noConv}
The set $R(D) / \eins$ does not contain any nontrivial usual convex sets.
\end{lemma}

\begin{proof}
Let $f, g \in R(D)$ be two tropical rational functions that are not translates of each other.  Then there is a smooth point $x \in \Gamma$ at which the slopes of $f$ and $g$ differ.  Let $0 < \lambda < 1$ be such that the convex combination $\lambda f + (1-\lambda) g$ has a non-integer slope at $x$.  Then $\lambda f + (1-\lambda) g$ is not even a tropical rational function, so it is not in $R(D)$.
\end{proof}
 
 Recall that $R(D)/\eins$, i.e. $R(D)$ modulo tropical scaling can be identified with the linear system $|D| := \{ D+(f) :  f \in R(D) \}$ via the map $f \mapsto D + (f)$.  In what follows, elements of $|D|$ and elements of projectivized $R(D)$ will be used interchangeably.
 
 A choice of model $G(\Gamma)$ induces a polytopal cell
decomposition of $\Sym^d \Gamma$. Andreas Gathmann and Michael
Kerber~\cite{GathmannKerber} as well as Grigory Mikhalkin and Ilia
Zharkov~\cite{MikhalkinZharkov} describe $|D|$ as a cell complex
$|D|_\Gamma \subset \Sym^d \Gamma$. Let us coordinatize this construction.

We identify each open edge $e \in E$ with the interval $(0,\ell(e))$ thereby giving the edge a direction, and  
$\Sym^ke$ with the open simplex $\{ x \in \RR^k \suchthat 0 < x_1 <
\ldots < x_k < \ell(e) \}$. 
A cell of $|D|$ is indexed by the following discrete data:
\begin{itemize}
\item $d_v \in \ZZ$ for every vertex $v \in V$,
\item an ordered composition $d_e = d_e^{(1)} + \cdots + d_e^{(r_e)}$
  for every edge $e$ of $\Gamma$, and
  \item an integer $m_e$ for every edge $e$ of $\Gamma$.
\end{itemize} 
Then, a divisor $D'$ belongs to that cell if
\begin{itemize}
\item $d_v = D'(v)$ for all $v \in V$,
\item $D'$ is given on $e$ by $\sum_i d_e^{(i)} x_i$ for $0 <
  x_1 < \ldots < x_{r_e} < \ell(e)$, and
\item the slope of $f$ at the start of edge $e$ is $m_e$, where $f$ is such that $(f)+D = D'$.
\end{itemize}
The intersection of $|D|$ with an open cell of $\Sym^d\Gamma$ is a union
of cells of $|D|$.

\begin{example}
Let $\Gamma$ be a circle (for example a single vertex $v$ with a loop edge $e$ attached).  Consider $D$ to be the divisor $3v$.  As we analyze in Example \ref{Circle3}, $|D|$ contains two $2$-cells in this case.  These two cells both contain a divisor $D'$ of the form $x + y + z$ with $x$, $y$, and $z$ points on the interior of $e$.  However the two-cells are differentiated from one another by looking at the slope of the function $f$ (defined by $D' = D + (f)$) at vertex $v$.  The outgoing slopes of $f$ at $v$ are given by $[-2,-1]$ and $[-1,-2]$ respectively for these two $2$-cells.  

This example shows that the combinatorial type of divisor $D'$ (without taking into account the slope of $f$) does not determine the combinatorial type of the corresponding cell.
\end{example}

This cell complex structure depends on the choice of the model $G(\Gamma)$, but not on the choice of representative divisor $D$ in the linear system $|D|$.  In particular, choosing a finer model amounts to subdividing the cell complex $|D|$.  Choosing a different divisor $D' = D + (g)$ amounts to changing the integer slopes at the starting points on the edges by the slopes of $g$, but this does not change the cells.

\begin{proposition} \label{cell-dim}
For $D' \in |D|$, and let $I_{D'}$ be the set of points in the support of $D'$ that lie in the interior of edges. Then the dimension of the cell containing $D'$ in its interior is one less than the number of connected components of $\Gamma \backslash I_{D'}$.
\end{proposition}

We assume that $\Gamma$ is connected, and being in the {\em interior} of an edge depends on the model $G(\Gamma)$.

\begin{proof}
If $\Gamma \backslash I_{D'}$ is connected, i.e.\ the chips on the interior of edges do not disconnect $\Gamma$, thus there is a connected subgraph of $\Gamma$ whose boundary points all lie in the interior of edges and have at least one chip at each these points.  
Hence, any weighted chip firing move will move chips on the vertices onto the interiors,  so there is no weighted chip firing move that preserves the combinatorial type of $D'$, so $D'$ is a vertex and has 
dimension zero. 

Suppose $\Gamma \backslash I_{D'}$ has $k$ connected components $\Gamma_1, \Gamma_2, \dots, \Gamma_k$, where $k \geq 2$.  
Let $C_i := \overline{\Gamma_i} \cap \overline{\Gamma \backslash \Gamma_i}$ be the set of points in the boundary of $\Gamma_i$ for each $i = 1, \dots, k$.  Points in $C_i$ are in $I_{D'}$, so they lie in 
the interiors of edges.  Then both $\Gamma_i$ and $\Gamma \backslash \Gamma_i$ can fire the chips in $C_i$ for a sufficiently small distance $\epsilon$ while preserving the combinatorial type.  
So $D'$ lies in the interior of these $k$ line segments.  These $k$ segments span a $k-1$ dimensional affine space, so the cell of $D'$ has dimension at least $k-1$. 
Moreover, the only subgraph that can fire $D'$ without changing the combinatorial types are the closures of union of $\Gamma_i$'s, and any tropical rational function is a sum of weighted chip firing moves.  So any point configuration in a neighborhood of $D'$ in the same cell can be attained as by a sequence of firing $\Gamma_i$'s, so the cell has dimension $k-1$.
\end{proof}

\begin{theorem}
Let $\cV$ be the set of vertices of the cell complex $|D|$ and $\cS(\cV) = \{f \in R(D): D + (f) \in \cV \}$.  Then
\begin{enumerate}
\item[(a)] $\cS(\cV)$ contains the set $\cS$ from Theorem \ref{thm:fingen},
\item[(b)] $\cS(\cV)$ is finite modulo tropical scaling, and 
\item[(c)] $\cS(\cV)$ generates $R(D)$ as a tropical semi-module.
\end{enumerate}
\end{theorem}

Recall that we identify divisors in $|D|$ with tropical rational functions in $R(D)$ modulo tropical scaling.

\begin{proof}
By the previous proposition, any element of $\cS$ has dimension $0$.  This shows (a), and (c) follows from (a) and Theorem \ref{thm:fingen}(c).  The statement (b) can be shown in the exact same way as Theorem \ref{thm:fingen}(b).
\end{proof}

If the model on $\Gamma$ is the coarsest one, i.e.\ the vertices are the points with valence other than two, then $\cV = \cS$.  If $\Gamma$ is a circle, then there is no coarsest model.

\begin{proposition}
Each closed cell in the cell complex is finitely-generated as a tropical semi-module by its vertices.  
\end{proposition}
 
 \begin{proof}
It is clear that each cell is closed under tropical scalar multiplication.
 Let $f_1, f_2$ be two rational functions in the same cell of the cell complex.  Then $f_1$ and $f_2$ have the same starting slopes on each edge, and $D + (f_1)$ and $D+(f_2)$ have the same number chips on each (open) edge.  Then the ending slope of each edge is also determined and equal for $f_1$ and $f_2$.  It is clear that $f_1 \oplus f_2$ have the same starting and ending slopes as $f_1$ and $f_2$, so it must also have the same number of chips as $f_1$ and $f_2$ on each edge.  So $f_1 \oplus f_2$ is in the same cell as $f_1$ and $f_2$.
 
To see the finite generation, we proceed as in the proof of Theorem \ref{thm:fingen}(c).  Let $f$ be a tropical rational function in $R(D)$.  If $\supp(D + (f))$ contains an interior cut set, then we can fire those in each direction as far as possible.  This amounts to writing $f$ as the tropical linear combination of two other functions $g_1, g_2$ in the boundary of the same closed cell.  By repeating the argument, $f$ can be written as a tropical linear combination of the vertices of the closed cell containing it.
 \end{proof}
 
 \begin{example}{(Line Segment)} Any tree is a genus zero tropical curve. Like genus zero algebraic curves, two divisors on a tree are linearly equivalent if and only if they have the same degree $d$.  The simplest tree is a line segment consisting of an edge $e$ between two vertices, say $v_1$ and $v_2$.  In this case, $|D|$ is a $d$-simplex.  The vertices of $|D|$ correspond to ordered pairs $[d_1, d_2]$ summing to $d$ associated to the chip configuration at $v_1$ and $v_2$. 
 \end{example}

 \begin{example}{(Circle)} \label{Circle3}
A circle is the only tropical curve where the canonical divisor $K$ is $0$.  Let $\Gamma$ be homeomorphic to a circle and let $D$ be of degree 3.  Then $D \sim 3 x$ for some point $x \in \Gamma$.  The coarsest cell structure of $R(D)$ is a triangle, but it is not realized by any model on $\Gamma$ because $\Gamma$ does not have a unique coarsest model.   If the model contains only one vertex $v$ and $D \sim 3 v$, then $R(D)$ is a triangle subdivided by a median; see Figure \ref{fig:Cycle3}.  In particular $|D|$ contains four $0$-cells, five $1$-cells, and two $2$-cells.  If the model $G(\Gamma)$ consists of a vertex $u$ such that $D \not\sim 3 u$, then the cell complex structure would be different.
 If the model $G(\Gamma)$ consists of $3$ equally spaced vertices $v_1, v_2, v_3$, and $D \sim 3 v_1$, then $R(D)$ is isomorphic as polyhedral complexes to the barycentric subdivision of a triangle. 
\end{example}

\begin{example}{(Circle with higher degree divisor)}
\label{ex:CircleHighDeg}
Let $\Gamma$ be a circle graph with only a single vertex $v$ and a single edge $e$, a loop based at $v$.  Let $D = dv$; then the linear system $|D|$ is a cone over a cell complex, which we denote as $P_d(circle)$, which has an $f$-vector
given by the following:
$$\mathrm{The~number~of~}i\mathrm{-cells~of~}P_d(circle) = f_i = (i+1) {d \choose i+2}.$$

Consequently, the $f$-vector for $|D|$ is given by
$$\begin{cases}
{d \choose 2} + 1 \mathrm{~if~} i = 0 \\
(i+1) {d \choose i+2} + i {d \choose i+1}  \mathrm{~if~} i\geq 1.
\end{cases}
$$

To see how to get these $f$-vectors, we note that a divisor $D' \sim dv$ corresponds to a tropical rational function $f$ such that $dv + (f) = D'$. One such $f$ is the zero function, this corresponds to the cone point. Each other tropical rational function is parameterized by an increasing sequence
of integer slopes $(a_1, \dots, a_{i+2})$ such that $a_1 < 0$, $a_{i+2} > 0$, and $a_{i+2} - a_1 \leq d$.  The first slope must be negative and the last slope must be positive so that the values of $f$ at the two ends of the loop $e$ agree.  The cells not incident to the cone point are given by sequences $(a_1,\dots, a_{i+2})$ such that all $a_i \not = 0$.  To finish the computation of the $f$-vector for the cell complex $P$ not incident to the cone point we pick an ordered pair $[j,k]$ with $j,k \geq 1$ and
$j+k=i+2$ to denote the number of negative and positive $a_k$'s, respectively. After setting $a_1 = -\ell$, we note that the number of ways to pick the remaining negative $a_k$'s is given by ${\ell-1 \choose j-1}$, and the number of ways to pick a subset of positive $a_k$'s such that $a_{i+2}-a_1 \leq d$ is given by ${d-\ell \choose k}$.  Summing over possible $\ell$, and using a standard identity involving binomial coefficients (for instance see \cite[Identity 136]{BQ}), we obtain ${d \choose i+2}$ such $f$'s for each $[j,k]$  Since there are $i+2$ such $[j,k]$'s, we get the above number of $i$-cells not incident to the cone point.  For the case of $d=4$, see Figure \ref{fig:Cycle4}.

\begin{figure}
\begin{center}
\includegraphics[scale=0.35]{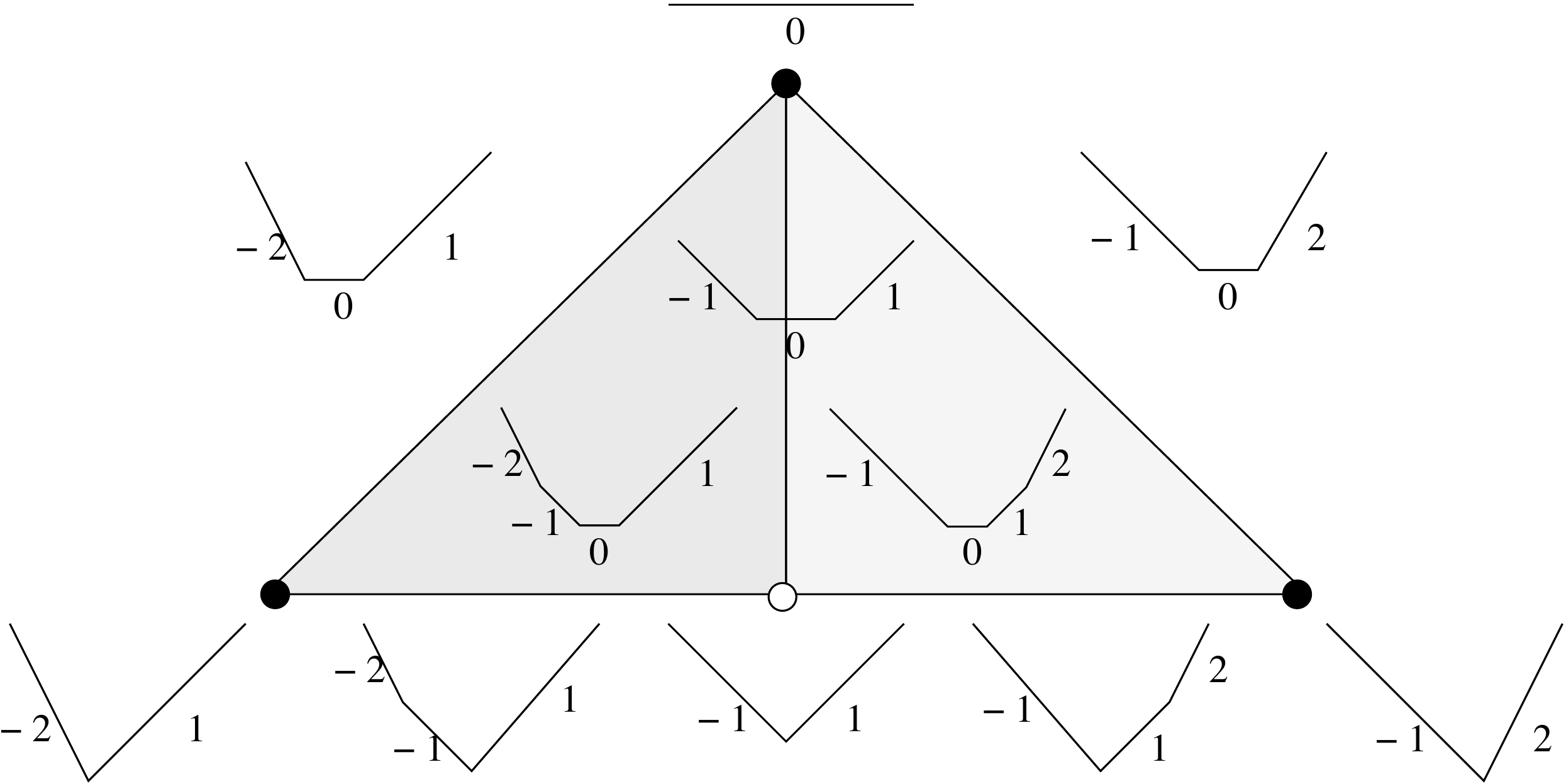}
\end{center}
\caption{The polyhedral cell complex $R(3v)/\eins$ on $\Gamma = S^1$.  The three black vertices are the extremals, and they correspond to the three divisors which are linearly equivalent to $3v$ and have the form $3w$.  We have presented $S^1$ as the line segment $[0,1]$ with points $0$ and $1$ identified.}
\label{fig:Cycle3}
\end{figure}

\begin{figure}
\begin{center}
\includegraphics[scale=0.35]{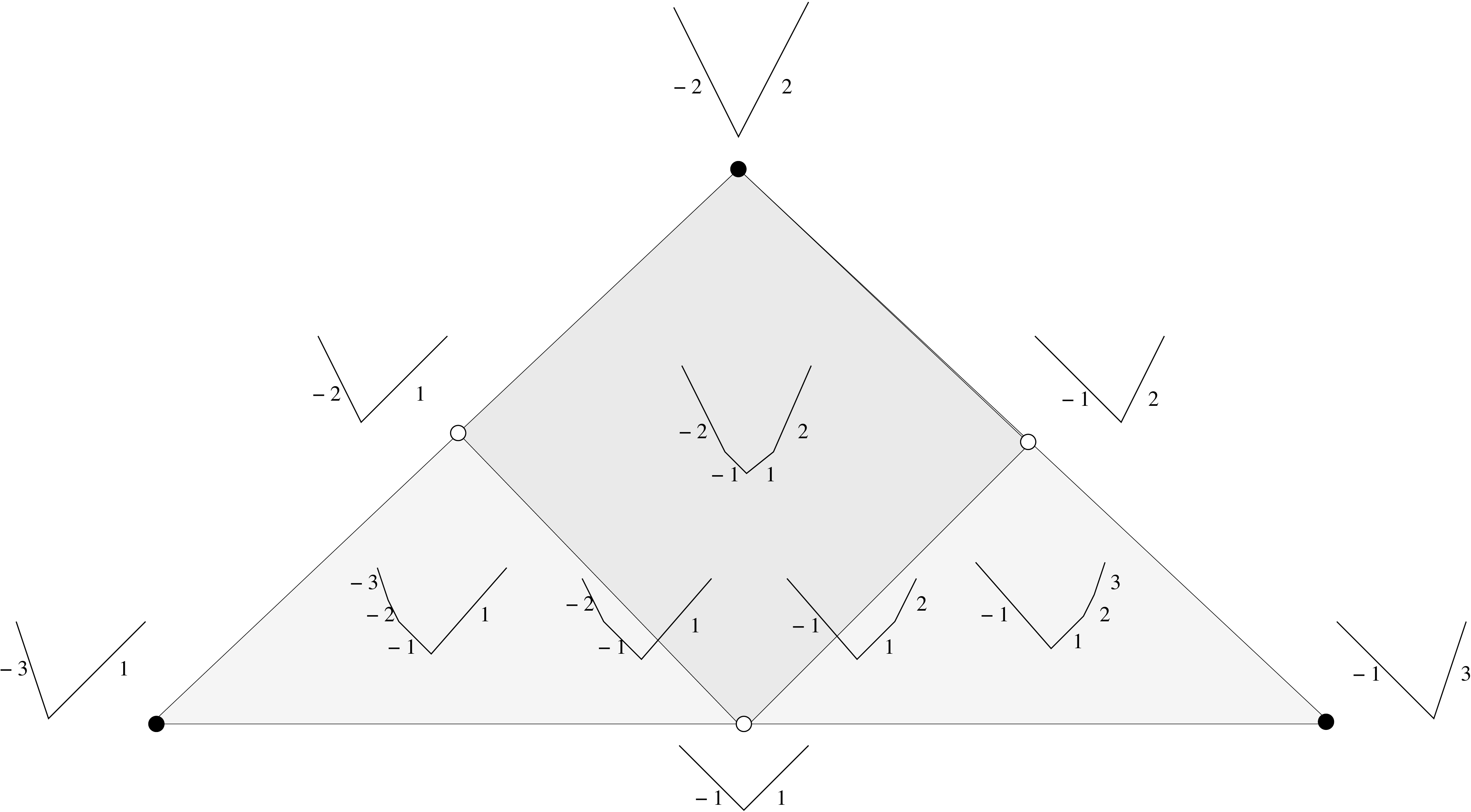}
\end{center}
\caption{The polyhedral cell complex $R(4v)/\eins$ on $\Gamma = S^1$ is a subdivided tetrahedron, a cone over this subdivided triangle with the cone-point corresponding to the constant function.  (The labels of most $1$-cells are suppressed, but may be read off from the incident vertices or $2$-cells.)  The cone-point plus the three black vertices are the extremals.}
\label{fig:Cycle4}
\end{figure}

\end{example}

 \begin{example}{(Circle.  Cell structure of $|D|$ as a simplex)}
 In Examples \ref{Circle3} and \ref{ex:CircleHighDeg}, we saw that having to choose a model, even one with only one vertex, gives $|D|$ a cell structure of a subdivided simplex.  Moreover, different choices of models, even if they contain only one vertex each, may give combinatorially different cell complex structures for $|D|$.  We wish to describe $|D|$ as a simplex.
 
First, let us look at the embedding of $|D|$ in the symmetric product of the tropical curve. 
 Let $\Gamma$ be the circle $\RR / \ZZ$, and $D = d \cdot [0]$ be a divisor of degree $d$.  
 The embedding of $|D|$ in
  $\Sym^d\Gamma = \Sym^d(\RR/\ZZ)$ is given by
$$\{ x \in (0,1]^d \suchthat 0 < x_1 \leq x_2 \leq \dots x_d \leq 1 , ~~  x_1+x_2 + \cdots + x_d \in \ZZ
\}.$$
To see this, first consider a tropical rational function $g$ on the line segment $[0,1]$ with 
$(g) = x_1+x_2 + \cdots + x_d - d \cdot \mathbf{0}$ and $g(\mathbf{1}) = 0$.  Then $g(\mathbf{0}) =  x_1+x_2 + \cdots + x_d$.  If $g(\mathbf{0}) \in \ZZ$, then adding $g$ and a function $l$ with constant slope $g(\mathbf{0})$  on $[0,1]$  gives a tropical rational function $f = g + l$ on the circle with $(f) + D = x_1+x_2 + \cdots + x_d$.  It is easy to check that any $f \in R(D)$ can be obtained this way.  Although this description gives $|D|$ a uniform coordinate system, this does not give us a cell complex structure.

In fact, $|D|$ can be realized as a $(d-1)$-dimensional simplex, on $d$ vertices.  There is a unique set of $d$ points $v_1, v_1, \dots, v_d$ in $\Gamma$ such that $D \sim d v_i$ for all $i = 1, \dots, d$.  These $d$ points are equally spaced along $\Gamma$.  The extremals of $R(D)$ are 
$$\mathcal{E} = \{f \in R(D) : (f) + D = d \cdot v_i \mbox{ for some } i = 1, 2, \dots, d\}.$$  
Consider the $(d-1)$-dimensional simplex on vertices $V = \{dv_1, dv_2, \dots, d v_d \}$, that is, the simplicial complex containing a $(k-1)$-dimensional cell for any $k$ subset of $V$.
We would like to stratify $|D|$ into these cells.  
For any divisor $D' \in |D|$, elements in the same cell as $D'$ are obtained from $D'$ by weighted chip firing moves that do not change the cyclically-ordered composition $d = a_1 + a_2 + \cdots + a_k$ associated to divisor $a_1 x_1 + a_2 x_2 +\cdots + a_k x_k$ where $x_1, x_2, \dots, x_k$ are distinct and cyclically ordered along the circle (with a fixed orientation).   The complement of the support of $D' = a_1 x_1 + a_2 x_2 +\cdots + a_k x_k$ consists of $k$ segments.  For each of these segments, there is a unique extremal in $R(D')$ that is maximal and constant on it.  These $k$ extremals of $R(D')$, which are naturally identified with extremals of $R(D)$, are precisely the vertices of the cell of $D'$ and their convex hull is the cell of $D'$. 
 \end{example}

\begin{example}{($K_4$ continued)}
\label{ex:cellsK4}
As in Example \ref{ex:gensK4}, consider the graph $K_4$ with same edge lengths and the canonical divisor.  The coarsest cell structure of $|K|$ consists of 14 vertices and topologically is the cone over the Petersen graph shown in Figure \ref{fig:gensK4}.  The cone point is the canonical divisor $K$.  The ``cones'' over the 
3 subdivided edges of the Petersen graph are quadrangles.  The maximal cells of $|K|$ consists of 12 triangles and 3 quadrangles.  In particular, $|K|$ is not simplicial.   The quadrangle obtained from ``coning'' over the bottom edge of the Petersen graph is shown in Figure \ref{fig:closeUpQuad}.
\end{example}

\begin{figure}
\begin{center}
\includegraphics[scale=0.35]{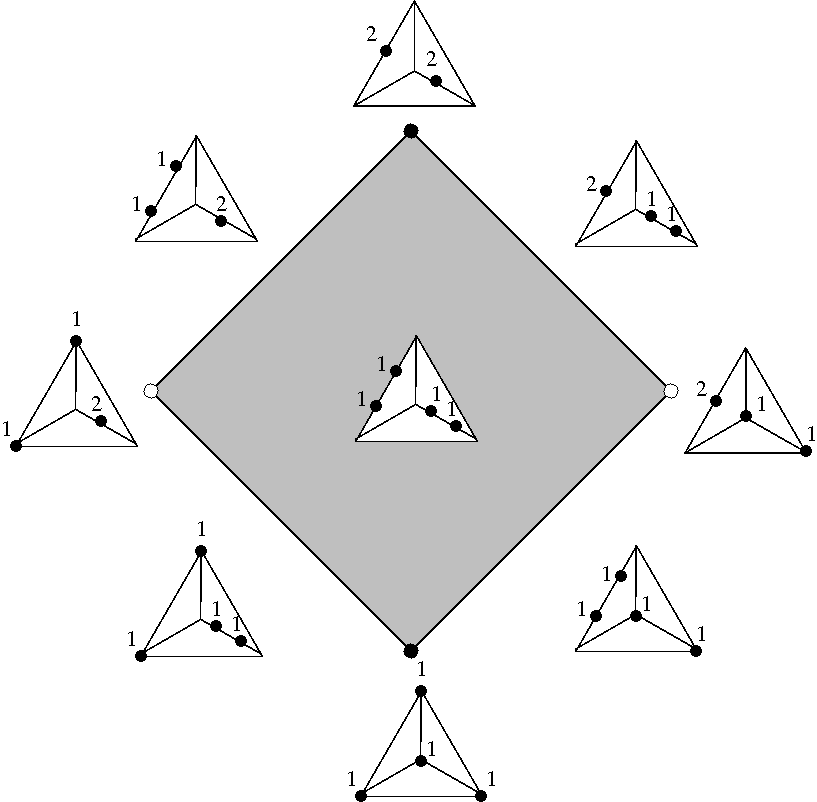}
\end{center}
\caption{A non-simplicial cell in the canonical linear system $|K|$ where $\Gamma$ is the complete graph on four vertices with edges of equal length.}
\label{fig:closeUpQuad}
\end{figure}
 
 \subsection{Local structure of a cell complex}
  
If $B$ is a cell complex and $x$ is a point in $B$, then the $\link(x,B)$ denotes the cell complex obtained by intersecting $B$ with a sufficiently small sphere centered at $x$.
We will define a triangulation of  $\link(D, |D|)$ which is finer than the cell structure.  Note that $|D|$ and $|D'|$ are isomorphic as cell complexes, so $\link(D, |D|) \cong \link(D, |D'|)$ for any $D' \sim D$.

Let $D' \in \link(D, |D|)$ and $f$ be a rational function such that $D' = D + (f)$.  Let $h_0 > h_1 > \cdots > h_n$ be the values taken on by $f$ on the set of points that are either vertices of $\Gamma$ or where $f$ is not smooth.  Notice that $h_0$ and $h_n$ are maximum and minimum values of $f$, respectively.  Since $D + (f) \in \link(D, |D|)$, we may assume that $h_0 - h_n$ is sufficiently small.  Let $G = (\Gamma_0 \subset \Gamma_1 \subset \cdots \subset \Gamma_n = \Gamma)$ be a chain of subgraphs of $\Gamma$ where $\Gamma_i = \{x \in \Gamma : f(x) \geq h_i\}$. 
 
 Let $G' = (\Gamma_1' \subset \Gamma_2' \subset \cdots \subset \Gamma_n' = \Gamma)$ be the chain of {\em compactified} graphs, where $\Gamma_i'$ is the union of edges of $\Gamma_i$ that are between two vertices of $\Gamma$.  Each cell can be subdivided by specifying more combinatorial data: the chain $G'$ obtained this way and the slopes at the non-smooth points.  We call this the {\em fine subdivision}.  
 
 For an effective divisor $D$, we can naturally associate the {\em firing poset} $\cP_D$ as follows.  An element of $\cP_D$ is a weighted chip firing move without the information about the length, i.e.\ it is a closed subgraph $\Gamma' \subset \Gamma$ together with an integer $c_e$ for each out-going direction $e$ of $\Gamma'$ such that for each point $x \in \Gamma'$ we have $\sum c_e \leq D(x)$ where the sum on the left is taken over the all outgoing directions $e$ from $x$ and $D(x)$ denotes the coefficient of $x$ in $D$.
 We say that $(\Gamma', c')  \leq (\Gamma'', c'')$ if $\Gamma' \subset \Gamma''$ and $c'_e \geq c''_e$ for each common outgoing direction $e$ of $\Gamma'$ and $\Gamma''$.
 
 \begin{theorem} \label{finesubdiv}
 The fine subdivision of the link of a divisor $D$ in its linear system $|D|$ is a geometric realization of the order complex of the firing poset $\cP_{D}$.
 \end{theorem}
 
 \begin{proof}
 By the discussion above, a cell in a fine subdivision $\link(D, |D|)$ corresponds to a unique chain in the firing poset.  
 For any chain in the firing poset, we can construct an element in $\link(D, |D|)$ by performing the weighted chip firing moves in the order given by the chain, starting from the smallest element.  The element constructed this way defines a cell in the fine subdivision.
 \end{proof}

Note that the link of an element in $|D|$ does not depend on the precise location of the chips, but on the combinatorics of the location.  In other words, changing the edge lengths, without changing which edges the chips are on, does not affect the combinatorics of the link.

This Theorem, along with Proposition \ref{cell-dim} allows us to explicitly describe the $1$-cells incident to a $0$-cell $D'$ of $|D|$.  For this, we need to define a specific subset of the weighted chip-firing moves.  In particular, we call a weighted chip-firing move $f$ (which is constant on $\Gamma_1$ and $\Gamma_2$) to be {\em doubly-connected} if $\Gamma_1$ and $\Gamma_2$ are both connected subgraphs. 

\begin{proposition}  \label{one-cells}
Given $D' \in |D|$, and a model $G$ such that $\supp(D') \subset V(G)$ (so that $D'$ is a $0$-cell in $|D|$), the $1$-cells incident to $D'$ correspond to the set of doubly-connected weighted chip-firing moves that are legal on chip configuration $D'$ (up to combinatorial type).
\end{proposition}

\begin{proof}
Let $f$ be a weighted chip-firing move which is legal at $D'$ that is constant on $\Gamma_1$ and $\Gamma_2$ such that $f(\Gamma_2) = f(\Gamma_1) - \epsilon$ for small $\epsilon > 0$.  Then $D''$, defined as $D' + (f)$ has a chip on each of the line segments $L_i$ connecting $\Gamma_1$ and $\Gamma_2$.  Then the dimension of the corresponding cell of $D''$ is one if and only if $\Gamma_1$ and $\Gamma_2$ are both connected.
\end{proof}

 \subsection{Bergman subcomplex of $|K|$}
 
Now we analyze the linear systems of an important family of divisors.  The {\em canonical divisor} $K$ on $\Gamma$ is $$K := \sum_{x \in \Gamma} (\val(x)-2)\cdot x.$$  Notice that since vertices of valence two do not contribute to the divisor $K$, the definition of $K$ does not depend on the choice of model. 
 
 Let $M$ be a matroid on a ground set $E$.  The {\em Bergman fan} of $M$ is the set of $w \in \RR^E$ such that $w$ attains maximum at least twice on each circuit $C$ of $M$.  
 The only matroids considered here are {\em cographic matroids} of graphs.  For a graph $G$ with edge set $E$, the cographic matroid is the 
  matroid on the ground set $E$ whose dependent sets are cuts of $G$, i.e.\ the sets of edges whose complement is disconnected.  
 The {\em Bergman complex} is the cell complex obtained by intersecting the Bergman fan with a sphere centered at the origin.  The following result will be useful to us later.
 
  \begin{theorem} \cite{AK}
  \label{thm:bergman}
 \begin{enumerate}
 \item The Bergman complex (with its fine subdivision) is a geometric realization of the order complex of the lattice of flats of $M$.
  \item  The Bergman fan is pure of codimension $\text{rank}(M)$.
 \end{enumerate}
 \end{theorem}

Note that adding or removing parallel elements does not change the simplicial complex structure of the Bergman complex because the lattice of flats remains unchanged up to isomorphism.  In particular, if $G_1$ and $G_2$ are two graphs, forming two models of the same tropical curve, then the corresponding cographic matroids have isomorphic Bergman complexes.

\begin{lemma}
\label{lem:circuits}
A subset of edges of a graph forms a flat of the cographic matroid if and only if its complement is a union of circuits of the graph.
\end{lemma}

\begin{proof}
The rank of a set $A$ of edges is one more than the size of $A$ minus the number of connected components of the complement of $A$.  Hence $A$ forms a flat if and only if there is no other edge outside of $A$ such that removing the edge increases the number of connected components of the complement.  This happens if and only if no connected component of the complement contain a one element cut set, i.e.\ the connected components are union of circuits.  Loops are considered circuits. 
\end{proof}

Suppose $\Gamma$ has genus at least one but $K_\Gamma$ is not effective.  Let $\Gamma'$ be the subgraph of $\Gamma$ obtained by removing all the leaf edges recursively.  Then the canonical divisor $K'$ of $\Gamma'$ is effective, and we can apply the following arguments for $K'$ in $\Gamma'$ or $\Gamma$.
 
 \begin{theorem}
The fine subdivision of $\link(K, |K|)$ contains the fine subdivision of the Bergman complex $B(M^*(\Gamma))$ as a subcomplex.  
\end{theorem}
  
 \begin{proof} 
The complement of a flat is a union of cocircuits, so the lattice of flats is isomorphic to the lattice of unions of cocircuits, ordered by reverse-inclusion.  The cocircuits of the cographic matroid are the circuits of the graph.   For the canonical divisor $K$, the proper union of circuits can always fire.  Hence proper part of the poset of union of circuits is a subposet of the firing poset, and so is the proper part of the lattice of flats.
 \end{proof}
 
The Bergman complex may be a proper subcomplex of the link because there may be subgraphs that can fire on the canonical divisor but that are not union of circuits, e.g.\ two triangles connected by an edge in the graph of a triangular prism.  Moreover, if $\Gamma$ is not trivalent, there may be vertices that can fire more than one chip on each edge, so the firing poset may be larger and the dimension of the order complex may also be larger.

\begin{example}
A family of graphs for which $\link(K, |K|)$ is isomorphic to the Bergman complex of the cographic matroid is as the following.  For an integer $n \geq 2$, consider a graph obtained from the cycle graph with $2n$ vertices and $2n$ edges by adding additional $n$ edges between pairs of opposite vertices in the cycle.  For this trivalent graph with the canonical divisor, the only subgraphs that can fire are unions of circuits.  Hence the firing poset is isomorphic to the lattice of unions of circuits, which is anti-isomorphic to the lattice of flats of the cographic matroid by Lemma \ref{lem:circuits}.  By Theorems \ref{finesubdiv} and \ref{thm:bergman}, the fine subdivision of $\link(K, |K|)$ is isomorphic to the Bergman complex of the cographic matroid.
 \end{example}
 
 \begin{example}{($K_4$ continued)}
\label{ex:linkK4}
 
 Let $\Gamma$ be a tropical curve with the complete graph on four vertices as a model, with arbitrary edge lengths.   Consider the canonical divisor $K$.    In this case, the firing poset coincides with the lattice of unions of circuits, which is anti-isomorphic to the lattice of flats.  Hence the link of the canonical divisor is isomorphic to the Bergman complex of the cographic matroid on the complete graph.  Since the complete graph on four vertices is self-dual, its co-Bergman complex is the space of trees on five taxa, which is the Petersen graph \cite{AK}.
 See Figure \ref{fig:linkK4}.
 
In the case when all edge lengths are equal, the quadrangles of $|K|$ described in Example \ref{ex:cellsK4} are subdivided in this fine subdivision of the $\link(K,|K|)$.   Note that the link of the canonical divisor stays the same when we vary the edge lengths, while the generators and cell structure of $R(K)$ in Figure \ref{fig:gensK4} may change.
 
 \begin{figure}
 \includegraphics[scale = 0.35]{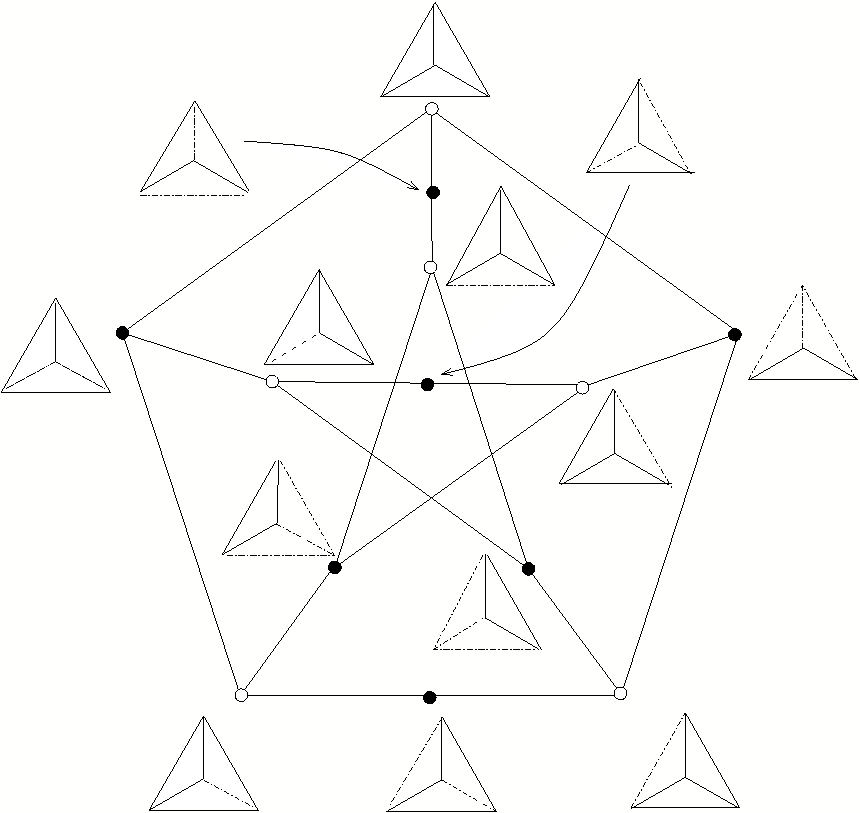}  
 \caption{ Link of the canonical divisor in the canonical class, where $\Gamma$ is the complete graph on four vertices, with arbitrary edge lengths.  Order complex of the firing poset.  The firing subgraphs in $\Gamma$ are shown by solid lines.
 See Example \ref{ex:linkK4}.  Compare with Figure 2 in \cite{AK}.}
 \label{fig:linkK4}
 \end{figure}
 
 \end{example}
 
\section{The induced map and projective embedding of a tropical curve}
\label{sec:Map}

A set $\F=(f_1, \ldots, f_r)$ of generators for $R(D)$ induces a map
$\phi_\F \colon \Gamma \to \TP^{r-1}$, defined as follows: For each $x \in \Gamma$, $\phi_\F(x) \mapsto (f_1(x), \dots, f_r(x))$.  This is a map into $\TP^{r-1}$ rather than $\RR^r$ as we consider $\F$ to be defined up to translation by $\eins$.  

\begin{theorem}
  The linear system $|D|$ (or $R(D)/\eins$) is homeomorphic to the tropical convex hull
  of the image of $\phi_\F$.
\end{theorem}

The {\em tropical convex hull} of a set is the tropical semi-module generated by the set.

\begin{proof}
  The intuition behind this theorem is the result
  from~\cite{DevelinSturmfels} that the tropical convex hull of the rows of a
  matrix is isomorphic to the tropical convex hull of the columns. Here, the matrix $M_\cF$ in
  question has entry $f_i(x)$ in row $i$ and column $x$.
  As in~\cite{DevelinSturmfels}, we define a convex set 
  $$P_\F = \{
  (y,z) \in (\RR^r \times \RR^\Gamma)/(\eins,-\eins) \suchthat y_i + z(x) \ge
  f_i(x) \}.$$
Let $B_\cF$ be the union of bounded faces of $P_\cF$, i.e.\ $B_\cF$ contains points in the boundary of $P_\cF$ that do not lie in the relative interior of an unbounded face of $P_\F$ in $ (\RR^r \times \RR^\Gamma)/(\eins,-\eins)$.
  We will show that $B_\cF$ projects bijectively
  onto $R(D)/\eins \subset \RR^\Gamma/\eins$ on the one hand, and to $\tconv \phi_\F(\Gamma)
  \subset \TP^{r-1}$ on the other, establishing a homeomorphism.

As in~\cite{DevelinSturmfels}, we associate a {\em type} to $(y,z) \in
  P_\F$ as follows:
    $$\type(y,z) := \{(i,x) \in [r] \times \Gamma: y_i + z(x) = f_i(x) \}.$$
In other words, a type is a collection of defining hyperplanes that contains $(y,z)$, so elements in the relative interior of the same face have the same type.
The recession cone of $P_\cF$ is $\{(y,z) \in (\RR^r \times \RR^\Gamma)/(\eins,-\eins) \suchthat y_i + z(x) \ge 0 \}$, which is the quotient of the positive orthant in $(\RR^r \times \RR^\Gamma)$ by $(\eins,-\eins)$.  Hence, a point $(y,z) \in P_\cF$ lies in $B_\F$ if and only if we cannot add arbitrary positive multiples of any coordinate direction to it while staying in the same face of $P_\cF$, which means keeping the same type.  This holds if and only if 
  \begin{enumerate}
  \item projection of $ \type(y,z) $ onto $[r]$ is surjective, and
  \item projection of $ \type(y,z) $ onto $\Gamma$ is surjective.
  \end{enumerate}
  For $(y,z) \in P_\F$, these two conditions are equivalent to 
    \begin{enumerate}
  \item $y_i = \max\{f_i(x) - z(x): x \in \Gamma\}$ , i.e.\ $y = M_\cF \odot -z$, and 
  \item $z(x) = \max\{f_i(x) - y_i : i \in [r]\}$, i.e.\ $z = -y \odot M_\cF$.
  \end{enumerate}
 where $M_\cF$ is the $[r] \times \Gamma$ matrix with entry $f_i(x)$ in row $i$ and column $x$, and $\odot$ is tropical matrix multiplication.  These two conditions respectively imply that the projections of $B_{\cF}$ onto $\RR^\Gamma/{-\eins}$ and $\RR^r/\eins$ are one-to-one.  The images are $R(D)/\eins$ and the tropical convex hull of $\image(\phi_\cF)$ respectively, so they are homeomorphic.
\end{proof}

\begin{remark}
All of the bounded faces of the convex set $P_\cF$ are in fact vertices.  If the union of bounded faces $B_\cF$ contained a non-trivial line segment, then its projection $R(D) / \eins$ would as well, contradicting Lemma \ref{lem:noConv}.
\end{remark}

\begin{example}[Circle, degree 3 divisor]
Let $\Gamma$ be a circle of circumference $3$, identified with $\mathbb{R} / 3 \ZZ$ and let $D$ be the degree 3 divisor $[0] + [1] + [2]$.  Let $f_0, f_1, f_2 \in R(D)$ be extremals corresponding to divisors $3\cdot [0], 3 \cdot [1],$ and $3 \cdot [2]$ respectively, and suppose $f_i ([i]) = -1$ for each $i=0,1,2$.  Then the image of $\Gamma$ under $\phi_\cF$, for $\cF = (f_0, f_1,f_2)$ is a union of three line segments between the points 
$$\phi_{\cF}([0])=(-1,0,0), ~~~ \phi_{\cF}([1]) = (0,-1,0), ~~~ \phi_{\cF}([2]) = (0,0,-1) ~~~ \mbox{ in } \TP^3 . $$
In this case, the (max-) tropical convex hull of the image of $\phi_\F$ coincides with the usual convex hull and is a triangle  However, it is not the tropical convex hull of any proper subset of $\image(\phi_\F)$.  In particular, $|D|$ is not a tropical polytope, 
 i.e.\ it is not the tropical convex hull of a finite set of points.
\end{example}

We know from \cite{DevelinSturmfels} that tropically convex sets are contractible.

\begin{corollary}
The sets $|D|$ and $R(D)$ are contractible.
\end{corollary}

Tropical linear spaces are tropically convex \cite{Speyer}, so any tropical linear space containing the image $\phi_\F(\Gamma)$ must also contain its tropical convex hull.

\begin{corollary}
Any tropical linear space in $\TP^{r-1}$ containing  $\phi_\F(\Gamma)$ has (projective) dimension at least the dimension of $|D|$.
\end{corollary}

\section{Embedded and Parameterized Tropical Curves}
\label{sec:Deg}

An {\em embedded tropical curve} $C$ is a one-dimensional polyhedral complex in $\TP^{n-1} = \RR^n / \eins$ with rational slopes, together with a {\em multiplicity} $m_e \in \ZZ_{>0}$ for each edge $e$ such that the following {\em balancing condition} holds.  For any point $x \in C$, we have $\sum{m_e v_e} = 0$ in $\TP^{n-1}$ where the sum is taken over all edges $e$ containing $x$ and $v_e \in \TP^{n-1}$ is the {\em primitive} integral vector in direction $e$ pointing away from $x$.  A vector $v \in \TP^{n-1}$ is called {\em primitive} if it has an integral representative in $\ZZ^n$ and generates the semigroup $(\ZZ^n \cap (\RR_+ v + \RR [\eins])) / \ZZ[\eins]$.

Suppose the embedded tropical curve $C$ has $k$ unbounded rays in primitive directions $v_1, \dots, v_k \in \TP^{n-1}$, with multiplicities $m_1, \dots, m_k$ respectively.  For each $v_i$, choose the unique representative in $\ZZ^n$ such that all the coordinates are non-positive, and the maximum coordinate is zero.  Because of the balancing condition, the sum $\sum_{i = 1}^k m_i v_i = - d \cdot \eins$ for a positive integer $d$.  This integer $d$ is called the {\em degree} of the embedded tropical curve $C$. 

A {\em parameterized tropical curve} is a tuple $(\Gamma, h)$ where
$\Gamma$ is an abstract tropical curve and $h : \Gamma \rightarrow
\TP^{n-1}$ is piecewise-linear with integer slopes and the following
{\em balancing condition} holds.  For any finite point $x \in \Gamma$,
$\sum m_e v_e = 0$ in $\TP^{n-1}$ where the sum is taken over all
out-going directions $e$ at $x$ in $\Gamma$, $v_e$ is primitive, $m_e$
is a positive integer, and $h : e \rightarrow \TP^{n-1}$ is the affine
linear map $h: t \mapsto t \, m_e \, v_e + h(x)$ where a neighborhood
of $x$ along $e$ is identified with the real interval $[0,l)$.

A point $x\in h(\Gamma)$ is called {\em smooth} if the fiber
$h^{-1}(x)$ is finite and consists of points at which $h$ is smooth.
The smooth points are dense in $h(\Gamma)$.  Let $e$ be an edge of
$h(\Gamma)$ and let $x$ be a smooth point on it.  Let the {\em
  multiplicity} of $e$ in $h(\Gamma)$ be $\sum_{y \in h^{-1}(x)} m_y$
where $h$ looks like $t \mapsto t\,m_y\, v + x$ locally along any
of the two out-going directions from $y$ and $v$ is primitive.  With
this definition of multiplicity, the image $h(\Gamma)$ is an embedded
tropical curve as defined earlier.  However, the statement that
$(\Gamma,h)$ is a parameterized tropical curve is stronger than the
statement that $h(\Gamma)$ is an embedded tropical curve.  The {\em
  degree} of a parameterized tropical curve $(\Gamma, h)$ is the
degree of the embedded tropical curve $h(\Gamma)$.

Let $\Gamma$ be an abstract tropical curve and $D$ be a divisor such that $R(D)$ is not empty and not equal to $\eins$. Choose a nonempty finite set of sections $\F = \{f_1, \dots, f_n\} \subset R(D)$; then we get a map $\phi_\F : \Gamma \rightarrow \TP^{n-1}$.  
In general, the tuple $(\Gamma, \phi_{\F})$ is not a parameterized tropical curve because it does not satisfy the balancing condition.

For $x \in \Gamma$, let $u(x) = \sum m_e v_e \in \TP^{n-1}$ where the sum is as in the balancing condition for parameterized tropical curves stated above.  Choose a representative of $u(x)$ in $\ZZ^n$ such that all the coordinates $u(x)_i$ are non-negative and the minimum coordinate is $0$.  A point $x$ satisfies the balancing condition if and only if $u(x) = 0$. 

For the next results we need the following definition, inherited from classical algebraic geometry.  
We say that a set of divisors $\mathcal{D} = \{D_i\}_{i \in \mathcal{I}}$ has {\em no base points} if for every $x \in \Gamma$, there is a $D' \in \mathcal{D}$ with $D'(x) = 0$.  If $D$ is understood, than we may say that a set of functions $\F$ has no base points if $\mathcal{D}(\F)$, defined as $\{D+(f) : f \in \mathcal{\F}\}$, has none.
  
\begin{lemma}\label{lem:unbalanced}
If $\F = \{f_1, \dots, f_n\}$ has no base points, then the coordinate $u(x)_i$ is the coefficient of $x$ in the divisor $D + (f_i)$.
\end{lemma}

\begin{proof}
Since each $f_i$ is in $R(D)$, we have $(D+(f_i))(x) \geq 0$.  Since
$\F$ has no base points, for every $x$,  $(D+(f_i))(x) = 0 $ for some
$f_i \in \F$.
By construction of $\phi_\F$, the $i^{th}$ coordinate of the vector
$m_e v_e$ in the sum is the outgoing slope of $f_i$ at $x$ along $e$.
Hence $u(x)_i = (f_i)(x)$, and
$$u(x)_i - u(x)_j = (f_i)(x) - (f_j)(x) = (D+(f_i))(x) - (D+(f_j))(x)
.$$
This proves the assertion.
\end{proof}

\begin{corollary}
If $\F = \{f_1, \dots, f_n\}$ has no base points, then the tuple
$(\Gamma, \phi_\F)$ satisfies the balancing condition as a
parameterized tropical curve if and only if the support of each
divisor $D+(f_i)$ contains only points at infinity for $i = 1, \dots,
n$.
\end{corollary}

Now we describe a natural way to extend $(\Gamma, \phi_\F)$ to obtain a parameterized tropical curve $(\tilde\Gamma, \tilde{\phi_\F})$.  Let $\tilde\Gamma$ be a tropical curve obtained from $\Gamma$ by attaching a leaf edge with infinite length at every finite point $x \in \Gamma$ at which the balancing condition is not satisfied.  Let $\tilde\phi_{\F} : \tilde\Gamma \rightarrow \TP^{n-1}$ be such that $\tilde\phi_\F |_\Gamma = \phi_\F$.  On the new unbounded leaf attached at $x$, identified with $[0, \infty]$, let $\tilde\phi_\F : t \mapsto - t \cdot u(x) + \phi_\F(x)$ where $u(x)$ is as in the lemma above.  By construction, $(\tilde\Gamma, \tilde\phi_\F )$ is a parameterized tropical curve.
The ``points at infinity'' are considered balanced and we do not attach new leaf edges to them.  Note that this {\em balancing procedure} depends not only on the image $\phi_\F(\Gamma)$ but on $\Gamma$ and $\F$ themselves.  In particular, this procedure may not yield the ``minimal'' embedded tropical curve containing $\phi_\F(\Gamma)$.  See Example \ref{ex:nonample}. However, it respects the structure of $\Gamma$ and $D$.

\begin{theorem}
\label{thm:deg}
The degree of the embedded tropical curve $\tilde\phi_\F(\tilde\Gamma)$ equals $\deg(D)$, for any base-point-free $\F \subset R(D)$.
\end{theorem}

\begin{proof}
First consider the case when the image $\phi_\F(\Gamma)$ is bounded in $\TP^{n-1}$.  Then all the unbounded rays of $\tilde\phi_\F(\tilde\Gamma)$ are images of the new leaf edges in $\tilde\Gamma$.  The new leaf edges are attached to the points $x\in \Gamma$ where $u(x) \neq 0$.   Such a point $x$ contributes $- u(x)$ to the computation of the degree of the embedded tropical curve $\tilde\phi_\F(\tilde\Gamma)$.  By Lemma \ref{lem:unbalanced}, the sum of such $- u(x)$ over all $x \in \Gamma$ is equal to $- \deg(D) \cdot \eins$.  So $\tilde\phi_\F(\tilde\Gamma)$ has degree $\deg(D)$ by definition.

Now suppose the image $\phi_\F(\Gamma)$ is not bounded.  Without loss of generality, we may assume that the support of $D$ contains no points at infinity.  Consider a new abstract tropical curve $\Gamma_0$ obtained from $\Gamma$ by truncating the unbounded leaf edges so that every $f_i$ is linear on $\Gamma \backslash \Gamma_0$ and $\supp(D)$ does not intersect  $\Gamma \backslash \Gamma_0$.  Then $\F_0 = \{ f_1|_{\Gamma_0} , \dots,  f_n|_{\Gamma_0}\}$ is a basepoint free subset of $R_{\Gamma_0}(D)$.  Moreover, 
$\tilde\phi_\F(\tilde\Gamma) = \tilde\phi_{\F_0}(\tilde\Gamma_0)$ and we are back to the bounded case.
\end{proof}

\begin{example}[Tropical lines]
Suppose $\Gamma$ is a tree and $deg(D) = 1$.   Let $v_1, v_2, \dots, v_n$ be the leaf vertices.  Let $f_i \in R(D)$ be the function such that $D + (f_i) = v_i$.  Then $\F = \{f_1, \dots, f_n\}$ gives an embedding $\phi_\F : \Gamma \hookrightarrow \TP^{n-1}$, and the balancing procedure produces a tropical line.  The image $\phi_\F (v_i)$ of the leaf vertex $v_i$ lies on the unbounded ray in $\tilde\phi_\F(\tilde\Gamma)$ pointing in direction $-e_i$. 
\end{example}

\begin{example}[Embedding a complete graph as the graph of a simplex]
Let $\Gamma$ be the the tropical curve obtained from the complete graph on $m$ vertices $v_1, v_2, \dots , v_m$ with all edge lengths 1, and consider the canonical divisor $K$ on $\Gamma$.  Let $f_i \in R(K)$ be the function such that $K + (f_i) = m \cdot v_i$.  For example, let $f_i = -1$ at $v_i$ and $f_i = 0$ on the edges not adjacent to $v_i$, and linearly interpolate. 
Although $\F = \{f_1, f_2, \dots , f_m\}$ does not generate $R(K)$, the map $\phi_\F : \Gamma \rightarrow \TP^{m-1}$ is an embedding.   We have $\phi_\F(v_i) = -e_i$, and all $f_i$ are linear on the interior of all edges, so the image $\phi_\F(\Gamma)$ consists of the edges of the simplex with vertices $\{-e_1, -e_2, \cdots, -e_m\}$.  The balancing procedure attaches an unbounded ray in direction $- e_i$ with multiplicity $m$ at the point $\phi_\F(v_i)$ for each $i \in \{1,2, \dots ,m\}$.  
\end{example} 

\begin{example}[$\Gamma = K_4$, $D=K$, $\phi_\F$ not injective]
\label{ex:nonample}

Let $\Gamma$ be as in the previous example.  Let $u_i \in \Gamma$ be the midpoints between $v_i$ and $v_4$, for $i = 1,2,3$.  See Figure \ref{fig:nonample}.  Let $f_1, f_2, f_3 \in R(K)$ be the functions such that $K + (f_1) = 2 u_1 + v_2 + v_3$, $K + (f_2) = 2 u_2 + v_1 + v_3$, and $K + (f_3) = 2 u_3 + v_1 + v_2$.  Then $\F = \{f_1, f_2, f_3\}$ is basepoint free, but $\phi_\F$ is not injective.  All of $v_1, v_2, v_3, v_4$ and the edges between $v_1, v_2, v_3$ are mapped to the same point under $\phi_\F$.  Their image is the point in the middle in the figure.  The image $\phi_\F(\Gamma)$ consists of the three solid line segments in the figure.  For a point in the interior of a solid line segment, the fiber has cardinality 2.  So each of three line segments has multiplicity 2.  

The tuple $(\Gamma, \phi_\F)$ is not yet a parameterized tropical curve because it does not satisfy the balancing condition at the points $u_1, u_2, u_3, v_1, v_2, v_3 \in \Gamma$, which are in the supports of $K+(f_i)$.  The point $v_4$ satisfies the balancing condition.  The balancing procedure produces a parameterized tropical curve with six unbounded rays as shown in Figure \ref{fig:nonample}.  Each of these has multiplicity 2.  From this, we see that the tropical curve has degree 4.

The embedded tropical curve $\tilde\phi_\F(\tilde\Gamma)$ we obtained this way from the balancing procedure is not a ``minimal'' embedded tropical curve containing $\phi_\F(\Gamma)$.  If we attached the unbounded rays only to the images of $u_1, u_2, u_3$, we would obtain a smaller embedded tropical curve of degree two containing $\phi_\F(\Gamma)$.

\begin{figure}
	\centering
	\scalebox{0.6}{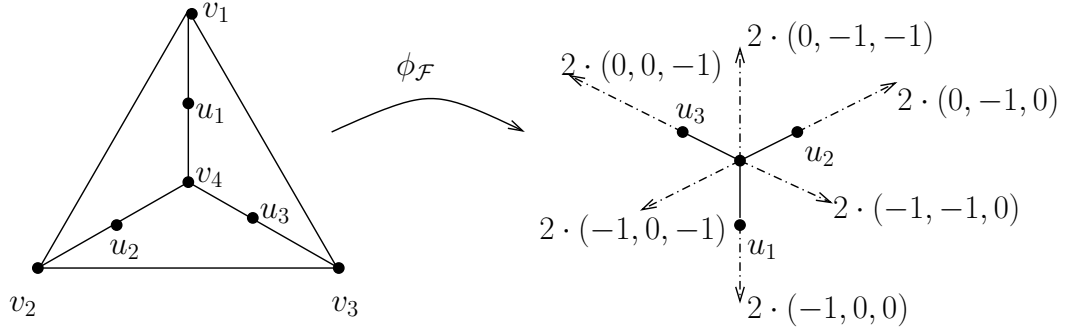}
	\caption{From Example \ref{ex:nonample}: $K_4$ with a non-injective map $\phi_\F$ to $\TP^2$.}
	\label{fig:nonample}
\end{figure}

\end{example}

\section{Canonical embeddings}
\label{sec:Canonical}

In this section we repeat the first steps in the classical theory of
ample divisors and canonical embeddings in the tropical
setting (compare~\cite[IV.3]{Hartshorne}). It turns out that for a few
classical equivalences, only one implication survives tropical
conditions.

Let $\Gamma$ be a tropical curve and $g$ its genus.
The {\em rank} $r(D)$ of a divisor $D$ is the maximum integer $r$ such that $|D - E| \neq \emptyset$ for all degree-$r$ divisors $E$.  The Riemann-Roch
Theorem~\cite{GathmannKerber,MikhalkinZharkov} (based on work of \cite{BakerNorine})
, which is the same for classical and tropical geometry, says the following: 
\begin{equation*}
  r(D)-r(K-D)=\deg D+1-g. \tag{RR}
\end{equation*}

We say that $D \ge 0$ is {\em very ample} if $R(D)$ separates points, that
is, if for every $x \neq x' \in \Gamma$ there are $f,f' \in R(D)$ with
$f(x)-f(x') \neq f'(x)-f'(x')$.
We call $D$ {\em ample} if some positive multiple $kD$ is
very ample. 

\begin{lemma}
  A divisor $D$ is very ample if and only if $\phi_\F$ is injective
  for any set $\F$ that generates $R(D)$. 
\end{lemma}

\begin{proof}
  The ``if'' direction is clear.  To see the ``only if'' direction,
  suppose there exist $x,x' \in \Gamma$ such that $f_i(x)-f_i(x') = f_j(x)-f_j(x')$ for all pairs $f_i, f_j
  \in \F$; then the same is true for any pair of tropical linear
  combinations of the $f_i$, and thus for all pairs $f,f' \in R(D)$.
\end{proof}

\begin{lemma}
  A divisor $D$ is very ample if and only if for all $x \neq x' \in
  \Gamma$ there is a $D' \in |D|$ whose support contains a smooth cut
  set separating $x$ and $x'$.
\end{lemma}

\begin{proof}
  Suppose $D$ is very ample with witnesses $f,f' \in R(D)$ for $x,
  x'$. Up to relabeling $c := f(x)-f'(x) - (f(x')-f'(x')) > 0$. 
  Then for a generic $\epsilon \in (0,c)$, the support of the divisor
  $D' := D + (f \oplus \epsilon \odot f')$ contains the smooth cut
  set $\{ x : f(x)-f'(x) = \epsilon \}$.

  Conversely, if the support of $D'=D+ (f) \in |D|$ separates $x$ from
  $x'$, we can use a cut function $g$ with $g(x)=f(x)<g(x')$ to
  construct $f' := f \odot g$.
\end{proof}

\begin{lemma} \label{nbp}
  If $r(D) \ge 1$ then $|D|$ has no base points.
\end{lemma}

\begin{proof}
  Let $x \in \Gamma$. Choose $x'$ on an edge incident to $x$. By
  assumption, there is a $D' \in |D|$ with $D'(x') > 0$. If $D'(x) >
  0$ then we can use $x'$ to pull $D'$ away from $x$, to get $D'' \in
  |D|$ with $D''(x) = 0$. Thus, $x$ is not a base point.
\end{proof}

The converse is false. Consider, for example, a curve $\Gamma$ of
positive genus with a {\em bridge} edge $e$ (i.e. an edge whose
removal disconnects $\Gamma$).    Then for $x \in e$, $|x|$ has no
base points, yet $r(x)=0$.

\begin{lemma}
  If $D$ is very ample, then $r(D) \ge 1$. In particular, very ample
  divisors have no base points.
\end{lemma}

\begin{proof}
  Let $x \in \Gamma$. Choose a sequence $x_n \in \Gamma \setminus
  \{x\}$ converging to $x$. There are divisors $D_n \in |D|$ so that
  $D_n$ contains a smooth cut set separating $x_n$ and $x$. Because
  $|D| \subset \Sym^{\deg D}\Gamma$ is compact, there is a converging
  subsequence. Its limit $D' \in |D|$ has $D'(x) > 0$.
\end{proof}

\begin{theorem} \label{thm:ample}
  If $\deg D \ge 2g+1$, then $D$ is very ample.
\end{theorem}

\begin{corollary}
  Every divisor of positive degree is ample.
\end{corollary}

\begin{proof}[Proof of Theorem~\ref{thm:ample}]
  As $\deg K = 2g-2$, $\deg(K-D) \le -1$, and $r(K-D)=-1$. By RR we
  have $r(D) \ge g+1$.

  Now, let $x \neq x' \in \Gamma$, and let $C = \{y_1, \ldots, y_r\}$ be a
  minimal smooth cut set separating $x$ and $x'$. Then $\Gamma \setminus C$
  has two connected components, $U,U'$, and the corresponding cut
  divisors agree: $D_U=D_{U'}=y_1+\ldots+y_r$. Furthermore, the genus
  of $\Gamma$ can be computed as $g =
  g(\overline{U})+g(\overline{U'})+r-1$. 
  In particular, $r(D) \ge g+1 \ge r$, and $|D-y_1-\ldots-y_r| \neq
  \emptyset$.

  Let $f \in R(D-y_1-\ldots-y_r)$, and set $D_0 := D+(f)$. By
  construction, $D_0 \ge D_U$, and we can use cut functions to
  separate $x$ from $x'$.
\end{proof}

A tropical curve is {\em hyperelliptic} if there is a
linear system $|D|$ with $\deg D=2$ and $r(D)=1$.  In particular, it is not a tree.

\begin{example}The curve with $2$ vertices joined by $g+1$ edges is
hyperelliptic. The following genus $3$ curve is hyperelliptic.
\begin{center}
  \includegraphics{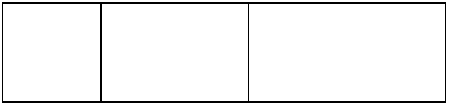}
\end{center}
\end{example} 

\begin{proposition}
  If $\deg D=2$, then $\phi_D(\Gamma)$ is a tree.  If in addition
  $r(D)=1$, then the fiber $\phi_D^{-1}(x) = \{y \in \Gamma :
  \phi_D(y) = x\}$ has size 1 or 2 for all $x$ in the image. 
\end{proposition}

\begin{proof}
A degree 2 tropical curve in $\TP^{n-1}$ cannot have cycles.  For $n = 3$, this follows from the facts that the polytope $2 \Delta_2 = 2 \cdot \text{conv}(e_1, e_2, e_3)$ in $\RR^3$ does not have any interior lattice points and that any degree 2 plane tropical curve is dual to a regular subdivision of a subset of $2 \Delta_2 \cap \ZZ^3$.  For higher $n$, this follows by induction and looking at projections.  By Theorem \ref{thm:deg}, the image  $\phi_D(\Gamma)$ can be extended to an embedded tropical curve of degree 2, so it must be a tree.

Now suppose $r(D) = 1$ and $x \in \Gamma$.  We have $\phi_D^{-1}(\phi_D(x)) = \{ x \}$ if $\Gamma \backslash x$ is disconnected, i.e.\ $x$ lies on a bridge.   Suppose $x$ lie in a cycle in $\Gamma$.  Since $r(D) = 1$ and $\deg(D) = 2$, there is a divisor $x + x' \in |D|$.  The point $x' \in \Gamma$ is the unique element of $|D - x|$ because of $x$ lies in a cycle.  Moreover $x'$ lies in every cycle that contains $x$; otherwise we would have $|D| = \{x+x'\}$, contradicting $r(D) = 1$.  Hence every cycle contains either both or none of $\{x, x'\}$.  Let $Z$ be a connected component of $\Gamma \backslash \{x, x'\}$.  Then $\val_Z(x) = \val_Z(x') = 1$; otherwise there would be a cycle that contains only $x$ or $x'$.  So there is a function $f \in R(x+x')$ such that $f(x)=f(x')=0$ but $f(q) < 0$ for any $q$ in the relative interior $Z^\circ$.  Together with the constant function $\eins \in R(x,x')$, $f$ separates $\{x,x'\}$ from $Z^\circ$.  The maps $\phi_D$ and $\phi_{x+x'}$ differ only by a translation,  so $\phi_D^{-1}(\phi_D(x)) \subset \{x, x'\}$. 
\end{proof}

\begin{theorem} \label{thm:canonical-embeddings}
  If $K$ is not very ample, then $\Gamma$ is hyperelliptic.
\end{theorem}

\begin{proof}
  Contracting a bridge edge does not affect either property.
  So it is safe to assume that $\Gamma$ has no leaf nodes.

  Suppose $K$ is not very ample, and let $x,y$ be two points that
  cannot be separated by $R(K)$. We claim that $r(x+y)=1$.
  The complement $\Gamma \setminus \{x,y\}$ splits into connected
  components $(X_i^\circ)_{i=1,\ldots,r}$,
  $(Y_j^\circ)_{j=1,\ldots,s}$, and $(Z_k^\circ)_{k=1,\ldots,t}$,
  whose closures satisfy
  $X_i \cap \{x,y\} = \{x\}$, 
  $Y_j \cap \{x,y\} = \{y\}$, and 
  $Z_k \cap \{x,y\} = \{x,y\}$.
  \begin{center}
    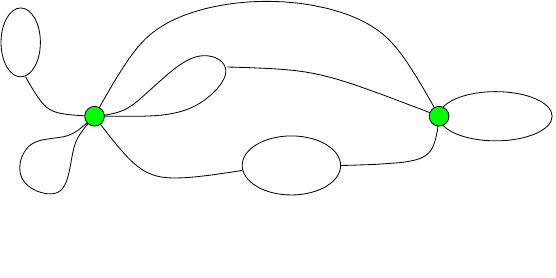
  \end{center}
  Because $r(K)=g-1$, the linear systems $|K-(g-1)x|$ and
  $|K-(g-1)y|$ are non-empty. Thus we must have
  $$\sum_k \val_{Z_k}(x) \ge g, \text{ and } \sum_k \val_{Z_k}(y) \ge g.$$ 
  Otherwise we can separate $x$ and $y$, contradicting our assumption.
  Because $Z_k^\circ$ is connected, the genus $g(Z_k)$ of $Z_k$
  satisfies
  $$ g(Z_k) \ge \val_{Z_k}(x) + \val_{Z_k}(y) -2\,.$$
  On the other hand,
  $$g = \sum_i g(X_i) + \sum_j g(Y_j) + \sum_k g(Z_k) + t - 1 \,.$$
  These relations imply $g-1 \le t \le g+1$.  Moreover, at most one of the
  $g(X_i), g(Y_j), g(Z_k)$ can be positive as we will see below.
  There are (up to symmetry) three cases.

  $\bullet$ \underline{$t=g+1$:}
  We must have $r=s=0$, and all $g(Z_k)=0$. Then $r(x+y)=1$.
  \begin{center}
    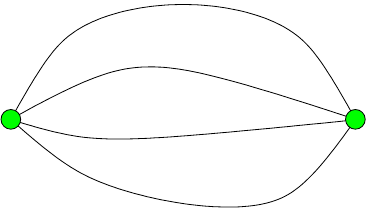
  \end{center}
  
  $\bullet$ \underline{$t=g$:}
  There are two subcases.

  \mbox{}\quad$\circ$ \underline{$r=1$, $s=0$:} (or $r=0$, $s=1$.) In
  this case, $|K|$ separates $x$ and $y$ as illustrated in the
  figure.
  \begin{center}
    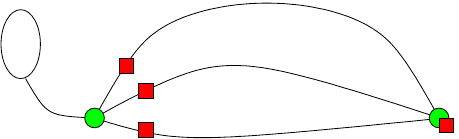
  \end{center}
  
  \mbox{}\quad$\circ$ \underline{$r=s=0$:} this is the only subtle
  case. One $Z_k$ has genus $1$. Its cycle has distance $\xi$ from $x$
  and distance $\eta$ from $y$.  At least one of the distances must be non-zero because $Z^\circ_k$ is connected.
  \begin{center}
    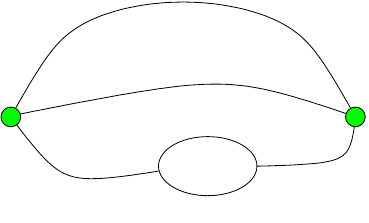
  \end{center}
  If $\xi \neq \eta$, then $|K|$ separates $x$ and $y$.
  \begin{center}
    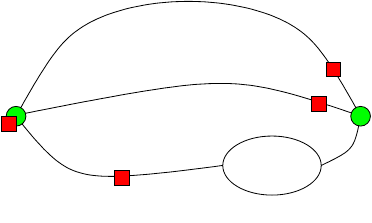
  \end{center}
  But if $\xi=\eta$, then $r(x+y)=1$.

  $\bullet$ \underline{$t=g-1$:} In this case, all inequalities in our
  chain have to be sharp. In particular, $r=s=0$. There is either one
  $Z_k$ of genus $2$ or two $Z_k$'s of genus $1$. Either way, $|K|$
  separates $x$ and $y$.
  \begin{center}
    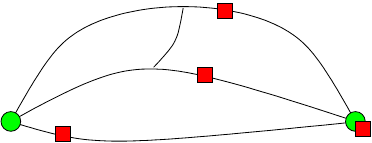 \qquad
    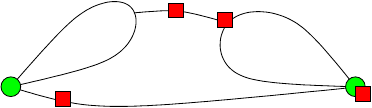 \qquad
    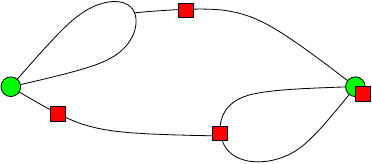
  \end{center}
\end{proof}
We have proved more than what the theorem states.
In fact, we can choose a divisor $x+y$ with $r(x+y)=1$ and $(g-1)(x+y)
\sim K$.

Sadly, the converse is false -- even if we take the stronger statement
above into account. For example, a flower with $\ge 3$ petals
satisfies the stronger condition, yet the canonical divisor is very
ample.
\begin{center}
  \includegraphics{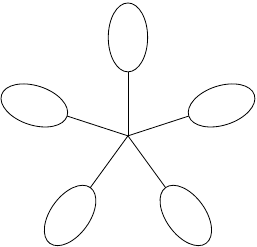}
\end{center}

The proof of Theorem~\ref{thm:canonical-embeddings} also yields the
following.
Here, a tropical curve is {\em generic} if maximal valency is three after contracting bridge edges, and the edge lengths are generic.

\begin{theorem}
  The canonical divisor of curves of genus $g \le 2$ is not very
  ample. In particular, such curves are hyperelliptic.
  Generic curves of genus $g \ge 3$ 
  are not hyperelliptic. In particular, their canonical divisor is very ample.
\end{theorem}

\begin{problem} Give a characterization of curves with
  $K$ not very ample.
  \end{problem}

\section{Tropical Picard Group and Continuous Chip-Firing}
\label{sec:Picard}

Following the classical definition of the Picard group in algebraic geometry, we define the \emph{tropical Picard group} of a tropical curve as a quotient group.  In particular, we take the group of degree zero $\Gamma$-divisors modulo the group  of principal divisors, i.e.\ those of the form $(f)$ where $f$ is a tropical rational function.  Before describing these groups further, we consider a finite graph analogue.

Given a finite undirected graph $G=(V,E)$ with $V= \{v_0,v_1,\dots, v_{n-1} \}$, 
we define the Laplacian matrix of $G$ to be 
$$L(G) = D(G) - A(G)$$ where $A(G)$ is the adjacency matrix of $G$ ($A_{ij} = $ the number of edges between $v_i$ and $v_j$) and $D(G)$ is the diagonal matrix $$D(G) = diag(\val(v_0), \val(v_1),\dots, \val(v_{n-1})).$$ We let $L_0(G)$ denote the reduced Laplacian matrix obtained by deleting the row and column corresponding to vertex $v_0$.

We define the {\em critical group} $K(G,v_0)$ (following \cite{B99} or \cite{D90}) to be the cokernel
$$K(G,v_0) = \ZZ^{n-1} /( L_0(G)~ \ZZ^{n-1}).$$
Choosing a different $v_0$ gives an isomorphic critical group.
For our calculations, we will use a set of explicit coset representatives of $K(G, v_0)$, known as 
\emph{superstable chip configurations}, following \cite[Section 4]{HLMPPW}:
  
We say that a divisor $D$ is a \emph{chip configuration} on graph $G$ \emph{with sink $v_0$} if the degree of $D$ is zero and $D(v) \geq 0$ for $v \not = v_0$.  We can also think of $D$ as a vector $[D(v_0), \dots, D(v_{n-1})]^T \in \mathbb{Z}^n$.

Given chip configuration $D$, we say that vertex $v \in V(G) \setminus \{v_0\}$ is \emph{ready to fire} if $D(v) \geq \val(v)$.  
The {\em chip firing move} given by the vertex $v$ is the map $\ZZ^n \rightarrow \ZZ^n$ given by $D \mapsto D - L_{v}$ where $L_v$ is the column of the Laplacian matrix $L(G)$ corresponding to the vertex $v$.  In other words, a chip firing move on a chip configuration moves a chip from $v$ to each of its neighbors.

We say that $D$ is \emph{stable} if for all $v \in V(G) \setminus \{v_0\}$, $0 \leq D(v) < \val(v)$, i.e.\  no vertex in $V(G) \setminus \{v_0\}$ is ready to fire.  We say that $v_0$ is ready to fire in $D$ if and only if $D$ is stable.  We say that a sequence $[v_{\alpha_1},v_{\alpha_2},\dots, v_{\alpha_\ell}]$ is \emph{legal} if $v_{\alpha_{i+1}}$ is ready to fire in $D^{(i)}$, the chip configuration resulting from firing the sequence $[v_{\alpha_1},v_{\alpha_2},\dots, v_{\alpha_i}]$ on $D$.  

We say that a \emph{cluster} $A \subset V(G) \setminus \{v_0\}$ can fire if the result of all vertices $v \in A$ firing simultaneously results in a chip configuration with all coefficients nonnegative.  (Note that it is possible for a cluster $A$ to fire even if no ordering of the elements of $A$ gives rise to a legal firing sequence.)  Finally, a configuration is said to be \emph{superstable} if no cluster can fire.  

We now use these explicit descriptions of elements of critical groups $K(G,v_0)$ to describe the tropical Picard group.  We get an explicit group structure on the set of superstable chip configurations by defining the sum of $D_1$ and $D_2$ to be  $\overline{D_1+D_2}$, the unique superstable configuration linearly equivalent to $D_1+D_2$.  

Following the terminology of \cite{GathmannKerber}, a $\QQ$-graph is a metric graph $\Gamma$ having a model $G$, for which each of the edges have a rational length.  An ordinary finite graph can be thought of as a $\QQ$-graph where all edge lengths are $1$.  Given a general metric graph (or tropical curve) $\Gamma$, we let $\Gamma_\QQ$ denote the set of points of $\Gamma$ whose distance from every vertex is rational.  We let $\Div_\QQ(\Gamma)$ denote the set of divisors on $\Gamma_\QQ$, also referring to these as the $\QQ$-divisor on $\Gamma$.  Further, $\Div^0_\QQ(\Gamma)$ will define the set of degree $0$ $\QQ$-divisors.

If $\Gamma$ is a $\QQ$-graph, we let $\Prin_\QQ(\Gamma)$ denote the principal $\QQ$-divisors, i.e. the subset of $\QQ$-divisors which are of the form $(f)$ for $f$ a tropical rational function.  We then define the $\QQ$-tropical Picard group of a $\QQ$-tropical curve to be $\Pic_\QQ(\Gamma)$ to be the quotient group $\Div^0_\QQ(\Gamma) / \Prin_\QQ(\Gamma)$.

\begin{theorem} \label{QPicard}
The $\QQ$-tropical Picard group of a $\QQ$-tropical curve $\Gamma$ is the direct limit of the critical groups corresponding to the subdivisions of $\Gamma$.
\end{theorem}

A more general version of this Theorem was independently proven by Baker and Faber \cite[Theorem 2.9 and 2.10]{BF2}.  Given $\QQ$-tropical curve $\Gamma$ we uniformly scale all edges to get $\Gamma'$ such that all edge lengths are integers.  We define $G_0$ to be the finite graph obtained by taking the coarsest finite graph structure on $\Gamma'$, with vertices given by points of valence one or $\geq 3$.  In the case of the cycle graph, all points are of valence two, so up to symmetry we define $G_0$ to be the graph with one vertex and one edge which is a loop at that vertex.  This is the unique example of $\Gamma$ with no points of valence one or $\geq 3$.

Without loss of generality, pick $v_0$ to be any vertex of $G_0$.
Let $G_k$ to be the model obtained by choosing more points as vertices such that all edges of $G_k$ have length $1/k$.  Since any vertex of $G_0$ is a vertex of $G_k$, it follows that $v_0$ is also a vertex of $G_k$ for all $k\geq 1$.  
Let $K(G_k,v_0)$ denote the critical group on subdivided graph $G_k$, using superstable configurations as coset representatives.

When $k_1$ divides $k_2$, we note that all vertices of $G_{k_1}$ are in  $G_{k_2}$, and thus we can define $\psi_{k_1,k_2}$ to be the map 
$$\psi_{k_1,k_2} : K(G_{k_1},v_0) \rightarrow K(G_{k_2},v_0)$$ sending the superstable configuration $D = \sum_{v \in G_{k_1}} D(v) \cdot v$ to $D' = \sum_{v \in G_{k_2} } D(v) \cdot v$.  Note here that the image, $D'$ has a coefficient of zero attached to any vertex $v \in G_{k_2}$ which is not in $G_{k_1}$. 

Since the valence of a vertex $v \in G_{k_1} \cap G_{k_2}$ is the same in both of these graphs, it follows that $D$ is a \emph{stable} configuration with respect to $G_{k_1}$ if and only if $\psi_{k_1,k_2}(D)$ is stable with respect to $G_{k_2}$.  We now wish to show that the superstability of $D$ with respect to $G_{k_1}$ implies the superstability of $\psi_{k_1,k_2}(D)$ with respect to $G_{k_2}$.  To see this equivalence, we note the following.

\begin{lemma} \label{TransitionMaps}
If $D$ is superstable with respect to graph $G_{k_1}$ and vertex $v_0$, then $\psi_{k_1,k_2}(D)$ is superstable with respect to graph $G_{k_2}$ and vertex $v_0$.  Furthermore, if $\overline{D_1 + D_2} = D_3$ in $K(G_{k_1},v_0)$, then $\overline{\psi_{k_1,k_2}(D_1) + \psi_{k_1,k_2}(D_2)} = \psi_{k_1,k_2}(D_3)$.
\end{lemma}

\begin{proof} If $D$ is \emph{superstable} with respect to $G_{k_1}$, then it follows that any cluster of vertices of $G_{k_1}$ outside $v_0$ cannot fire.  In particular, for all subsets $S \subset V(G_{k_1})$, it follows that there exists at least one $v \in S$ such that, $\mathrm{outdeg}_S(v) > D(v)$. 

The configuration $\psi_{k_1,k_2}(D)$ has the same support as $D$ by construction.  However, to check superstability of $\psi_{k_1,k_2}(D)$, we must consider all clusters of vertices in $G_{k_2} \setminus \{v_0\}$.  Consider any subset $S \subset V(G_{k_1}) \setminus \{v_0\}$, such a subset is also a subset of $V(G_{k_2}) \setminus \{v_0\}$.  For any $v \in V(G_{k_1})$, the outdegree of $v$, with respect to $S$ is the same in $G_{k_2}$ as it was in $G_{k_1}$ so such subsets $S$ cannot fire in $V(G_{k_2})$.  Any subset $S \in V(G_{k_2})$ formed by adjoining all vertices along a subdivided edge similarly cannot fire.  Finally, any subset $S$ which contains a new degree $2$ vertex $v'$ in $V(G_{k_2})\setminus V(G_{k_1})$ but not the two neighbors of $v'$ cannot fire since $\mathrm{outdeg}_S(v') > 0 = D(v')$ for such subsets.

Lastly, if $\overline{D_1+D_2} = D_3$, then $\psi_{k_1,k_2}(D_1+D_2) = \psi_{k_1,k_2}(D_1) + \psi_{k_1,k_2}(D_2)$.  We also have $\psi_{k_1,k_2}(\overline{D_1+D_2}) = \overline{\psi_{k_1,k_2}(D_1+D_2)}$ since for any $D_3$, $\psi_{k_1,k_2}(\overline{D_3})$ is superstable and is linearly equivalent to $\psi_{k_1,k_2}(D_3)$ so $\psi_{k_1,k_2}(\overline{D_3})$ must equal $\overline{\psi_{k_1,k_2}(D_3)}$.
\end{proof}

\begin{proof} [Proof of Theorem \ref{QPicard}] From Lemma \ref{TransitionMaps}, we see that  the transition maps preserve superstability and are compatible with addition.  Furthermore, the transition maps are injective and satisfy $\psi_{k_1,k_3} = \psi_{k_2,k_3} \circ \psi_{k_1,k_2}$ when $k_1 | k_2 | k_3$.  We use these transition maps to define the direct limit $$\overline{K}(\Gamma,v_0) := 
\mathop{\lim_{\longrightarrow}}_{k\geq 1}~ \{K(G_k,v_0)\} = 
\bigcup_{k=1}^\infty K(G_k,v_0) \bigg / \sim$$
where $D \in K(G_{k_1},v_0)$ and $D' \in K(G_{k_2},v_0)$ are equivalent if and only if $\psi_{k_1,k_3}(D) = \psi_{k_2,k_3}(D')$ in $K(G_{k_3},v_0)$ for $k_3 = \lcm(k_1,k_2)$.  Since the $\psi_{k_i,k_j}$'s are injective group homomorphisms, for all $k\geq 1$, $\overline{K}(\Gamma,v_0)$ contains a subgroup isomorphic to $K(G_k,v_0)$.

Thus the direct limit group is well-defined and it is direct to verify that its definition agree with that of the $\QQ$-tropical Picard group. 
In particular, we can think of this direct limit as the critical group on the infinite graph with countably infinitely many vertices, one vertex for each (rational) point of $\Gamma_\QQ$.  Since the divisors in the image of the Laplacian matrix are exactly those divisors which are of the form $(f)$, it follows that the tropical Picard group of a $\QQ$-tropical curve $\Gamma$ is the direct limit of critical groups as claimed.
\end{proof}

One advantage of this description of the $\QQ$-tropical Picard group on $\Gamma$ as a direct limit of finite critical groups is that we can simulate (weighted-) chip-firing moves on $\Gamma$ by looking at the limit of firing sequences on the finite subdivisions.  

For all vertices $v \in G_{k_1}$, we define $H_0(v) = \{v\}$, let $N(H)$ denote the neighbors of graph $H$, and inductively define $H_i(v)$ as $$H_i(v) = H_{i-1}(v) \cup N(H_{i-1}(v)).$$ 

In other words, letting $m = \frac{k_2}{k_1}$,  $H_m(v)$ is the unique induced subgraph of $G_{k_2}$ which is a tree with root $v$ and leaves given by neighbors of $v$ in unsubdivided graph $G_{k_1}$, and $\{v_0\} = H_0(v) \subset H_1(v) \subset H_2(v) \subset \dots \subset H_{m-1}(v)$ is a chain of subgraphs;  see Figure \ref{fig:Hk}.

\vspace{1em}

\begin{figure}
	\centering
	\scalebox{0.4}{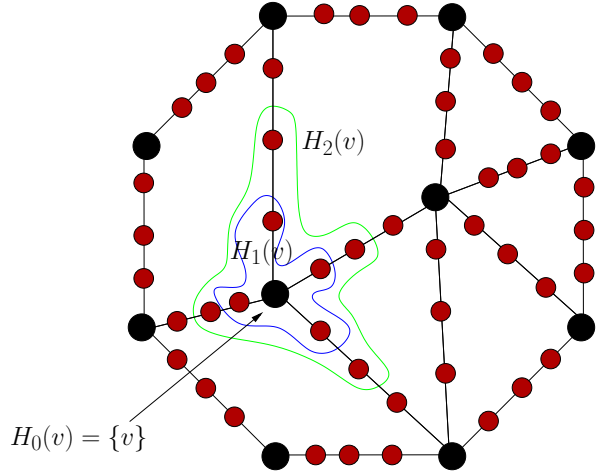}
	\caption{The nested subgraphs $H_0(v)=\{v\} \subset H_1(v) \subset H_2(v) \subset \dots \subset H_{m-1}(v)$.}
	\label{fig:Hk}
\end{figure}

Using this notation, we get a method for emulating firing sequences of vertices in $V(G_{k_1}\setminus \{v_0\})$

\begin{lemma} \label{emulate}
Assume for all $0 \leq i \leq m-1$ that we fire subset $H_i(v)$ by firing vertex $v$ first and then radiating outwards towards the leaves.  If $D_{(k_1)}$ is a chip configuration on $G_{k_1}$ such that vertex $v$ is ready, then the firing of vertex $v$ in graph $G_{k_1}$ can be emulated on chip configuration $D_{(k_2)}$ in graph $G_{k_2}$ by firing sequence 
$$[H_{m-1}(v),H_{m-2}(v),\dots, H_2(v), H_1(v), H_0(v)].$$
\end{lemma}

\begin{proof}
We assume $ m = k_2/ k_1 \geq 2$ since the statement is trivial otherwise.  Since vertex $v$ can fire the coefficients $D_{(k_1)}^*(v) = D_{(k_2)}^*(v)  \geq \val(v)$ where this valence is the same in both graphs $G_{k_1}$ and $G_{k_2}$.  Thus $v$ can fire in $G_{k_2}$.  Since its neighbors are valence two vertices which were not in $G_{k_1}$, they each start with exactly one chip on them, so after vertex $v$ fires, they have two chips on them and thus they in turn can fire.  Continuing in this way, we have a cascading effect, and the end result of firing subgraph $H_{m-1}(v)$ is the addition of one chip to each neighbor $w_j$ of vertex $v$ (with respect to graph $G_{k_1}$) and the subtraction of one chip from the leaves of tree $H_{m-1}(v)$.  See Figure \ref{fig:Hk-firing}.  

By analogous logic we can now fire $H_{m-2}(v), H_{m-3}(v), \dots , H_0(v)$ in turn, which results in a configuration identical to the one we obtain by firing vertex $v$ in graph $G_{k_1}$.
\end{proof}

\begin{example} [Circle]
\item The Picard group of a circle $\Gamma = S^1$ with length $\ell$ is the circle group $\RR / {\ell ~ \ZZ} \cong S^1$.  In particular, if we look at the finite model $G = (\{v\}, \{e\})$, where $e$ is a loop edge, and then subdivide $e$ into $k$ equal segments, we obtain $G_k = C_k$, the $k$-cycle graph.  It is direct to verify that $\Pic(G_k) \cong \ZZ / k \ZZ$ with superstable representatives given by a divisor of the form $(w)$ or $0$, where $w$ is any vertex of $\Pic(G_k)$.  The group law then exactly matches the group law on $S^1$ so we get the appropriate direct limit.
\end{example}

\begin{example} [Genus 2 Banana Graph]
Let $\Gamma$ be the metric graph with two vertices ($v_1$ and $v_2$) connected by three parallel edges.  We can use finite models of this graph to gain intuition for the structure of $\Gamma$'s tropical Picard group.  Let $G_k$ be the finite graph on $3k+2$ vertices defined as the union of the three path graphs $[v_1, x_1,x_2,\dots, x_k, v_2]$, $[v_1, y_1,y_2,\dots, y_k, v_2]$, and $[v_1, z_1,z_2,\dots, z_k, v_2]$, see Figure \ref{fig:banana}.  By direct verification we see that $\Pic(G_k) \cong \ZZ / (k+1) \ZZ \times \ZZ / (3k+3) \ZZ$ using the generators $D_1 = x_1 - v_1$ and $D_2 = y_2 + z_1 - 2v_1$.  (If $k =1$, we use $D_2 = v_2 + z_1 - 2v_1$ instead.)  

\begin{figure}
	\centering
	\scalebox{0.5}{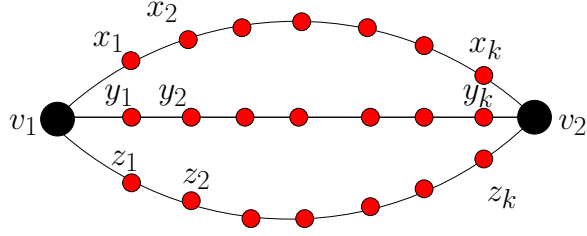} 
	\caption{Subdivided genus $2$ Banana graph.}
	\label{fig:banana}
\end{figure}

In particular multiples of $D_1$ have the form $$x_m - v_1, ~~ v_2 - v_1, ~~ x_m + v_2 - 2v_1, ~~ 2v_2 - 2v_1, \mathrm{~~or~~} y_{k+1-m} + z_{k+1-m} - 2v_1,$$ in order.  Multiples of $D_2$ look like 
$$y_{2m} + z_m - 2v_1,~~ v_2 + z_{\frac{k+1}{2}} - 2v_1, \mathrm{~~or~~} x_{k+1 - 2m} + z_{\frac{k+1}{2}-m} - 2v_1,$$ in order, if $k$ is odd, and the case where $k$ is even is analogous.  Taking the direct limit, we thus see that $\Pic_{\QQ}(\Gamma) \cong \QQ/\ZZ \times \QQ/\ZZ$ generated by divisor classes $\overline{D_1}$ and $\overline{D_2}$ where sample representatives are of the form $$D_1 = x_\alpha - v_1 \mathrm{~and~} D_2 = y_{2\beta} + z_\beta - 2v_1.$$  

\begin{figure}
	\centering
	\scalebox{0.3}{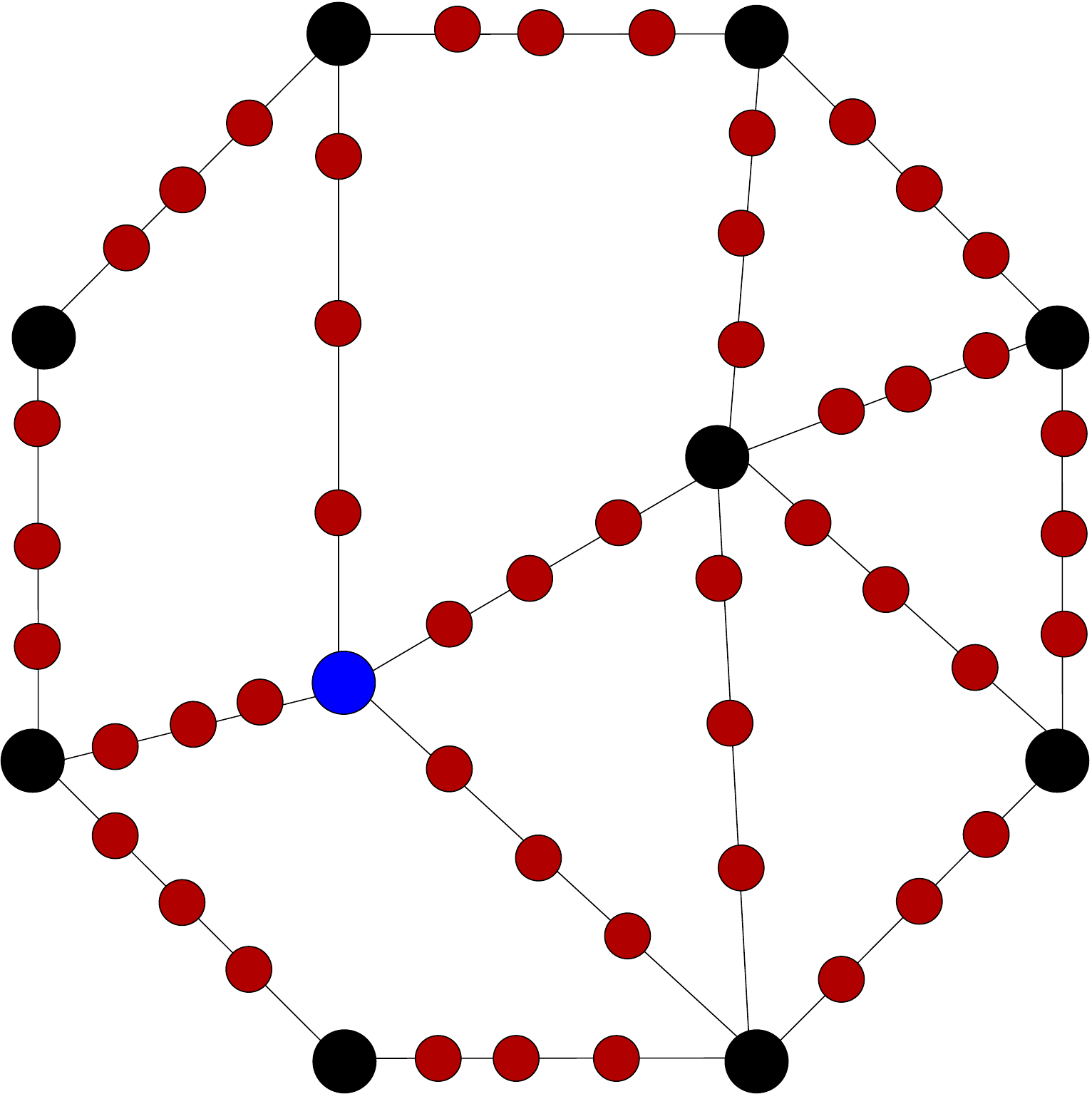} \hspace{2em}
	\scalebox{0.25}{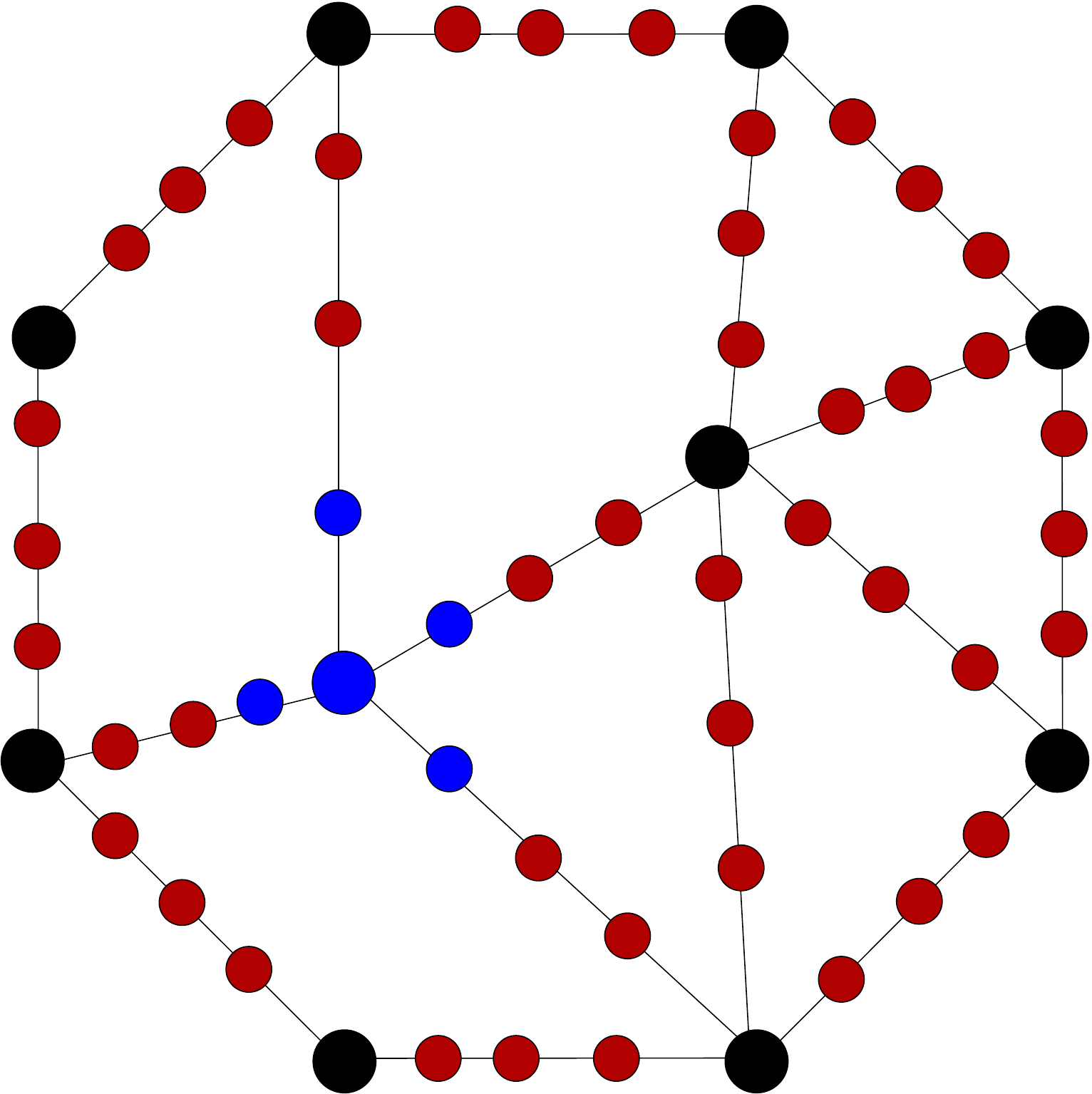} \hspace{2em}
	\scalebox{0.25}{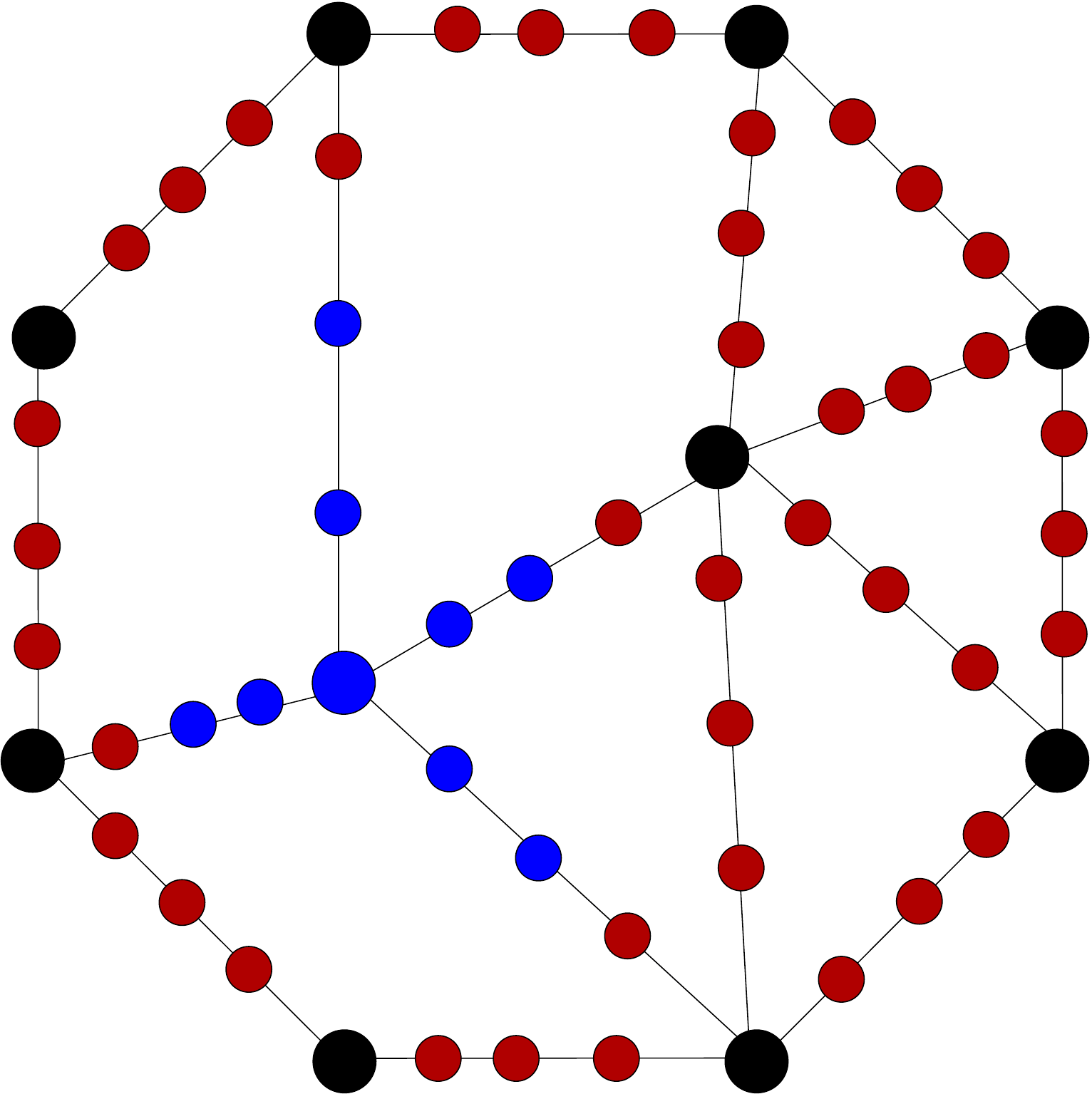}  \hspace{2em}
	\scalebox{0.25}{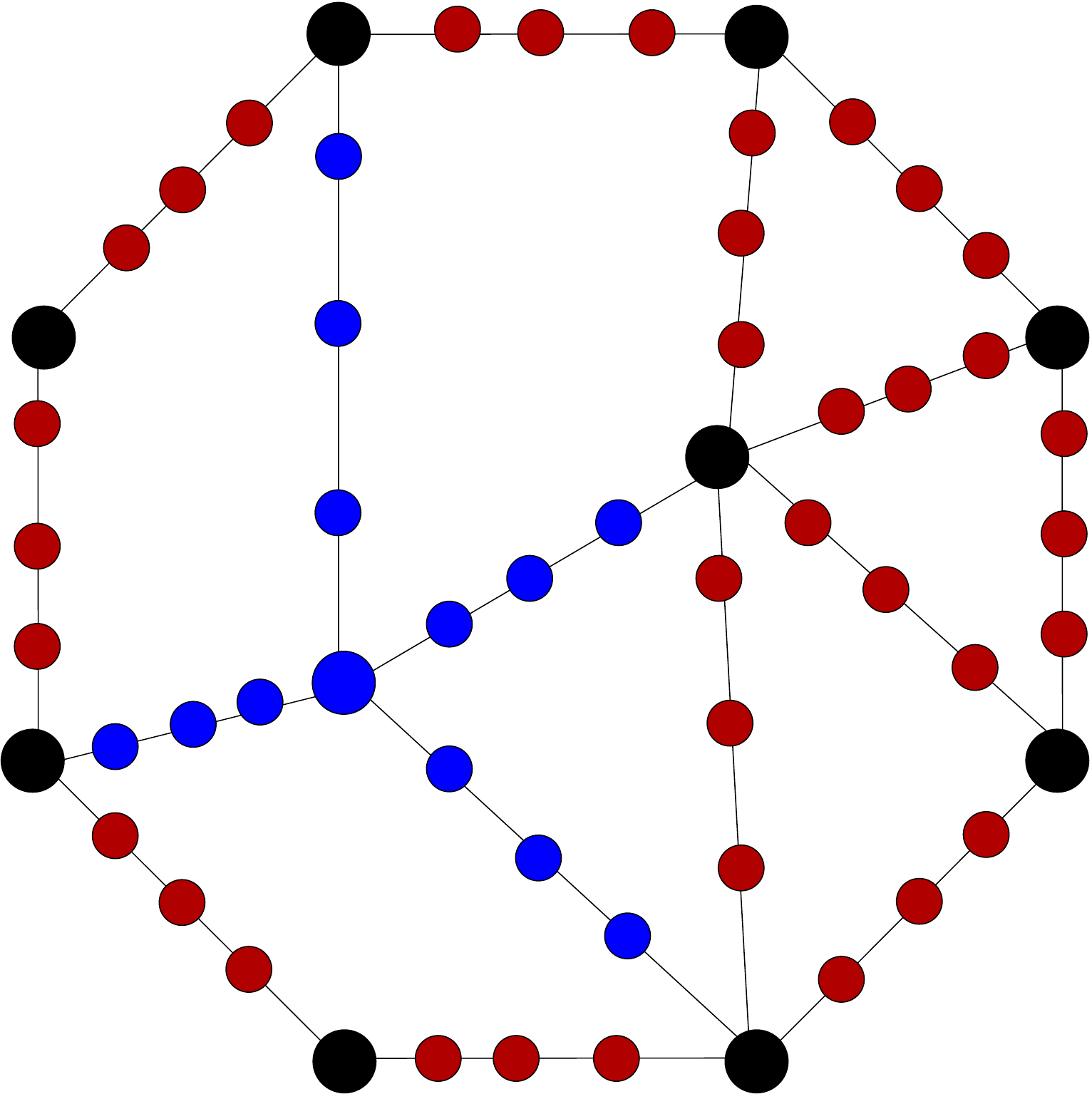} \hspace{2em}
	\scalebox{0.3}{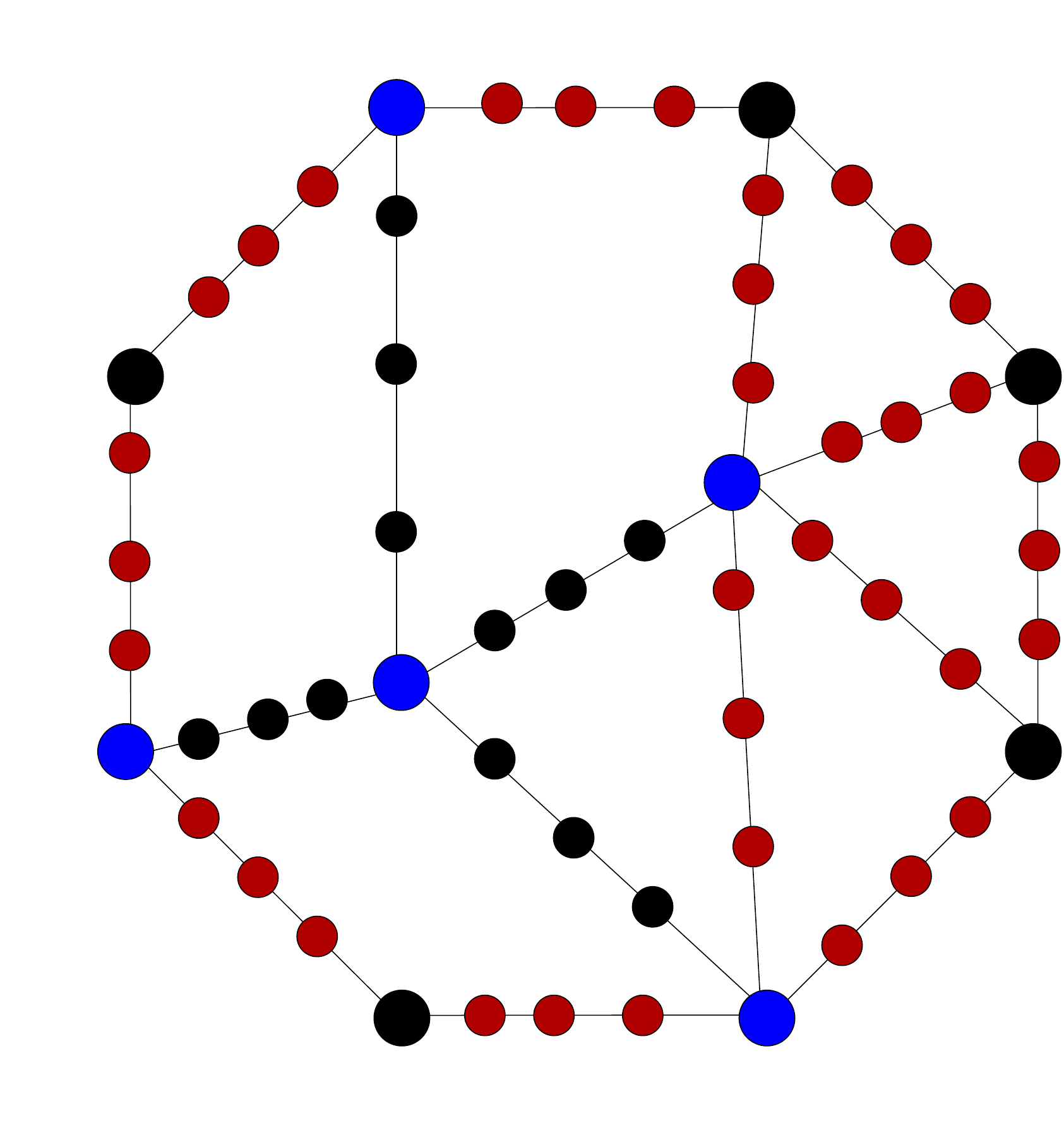}
			
	\caption{Simulating chip-firing at $v$ using a sequence of firing nested subgraphs on subdivided graph.}
	\label{fig:Hk-firing}
\end{figure}

\noindent Here $x_\alpha$ denotes the point on the top edge distance $\alpha$ from vertex $v_1$.  $y_\alpha$ and $z_\alpha$ are defined analogously.  We can also see that extending from rational points of $\Gamma$ to real points can be accomplished in this case simply by letting distances $\alpha$ and $\beta$ be real.     
\end{example}

\section{Conclusions and Open Questions}

In this paper, we presented a number of properties of $|D|$ including verification that it is finitely generated as a tropical semi-module.  We also provided 
some tools for explicitly 
understanding $|D|$ 
as a polyhedral cell complex such as a formula for the dimension of the face containing a given point, as well as 
applications such as using $|D|$ to embed an abstract tropical curve into tropical projective space. 

There are many ways to continue this research for the future.  It is quite tantalizing to investigate how the Baker-Norine rank of a divisor compares with the geometry and combinatorics of the associated linear system as a polyhedral cell complex.  Also, is there any relation between $r(D)$ and the minimal number of generators of $R(D)$?  Can we identify out of our finite generating set $\cS$ which $0$-cells correspond to extremals?  How does the structure of $|D|$ change as we continuously move one point in the support of $D$ or if we change the edge lengths of our metric graph while keeping the combinatorial type of the graph fixed? 

In the case of finite graphs, i.e. divisors whose support lies within the set of vertices of the graph, can we combinatorially describe the associated linear systems?  For example, is there a stabilization or an associated Ehrhart theory that one can use to count the sizes of such linear systems?  
 
Lastly, what other results from classical algebraic geometry carry over from the theory of metric graphs (or tropical curves) and vice-versa?

 \end{document}